\documentclass[10pt, reqno, final]{amsart}

\usepackage[T1]{fontenc}
\usepackage[utf8]{inputenc} 
\usepackage{baskervald} 

\usepackage{microtype} 

\makeatletter
\def\MT@register@subst@font{\MT@exp@one@n\MT@in@clist\font@name\MT@font@list 
  \ifMT@inlist@\else\xdef\MT@font@list{\MT@font@list\font@name,}\fi}
\makeatother

\usepackage[letterpaper, showframe]{geometry}
\usepackage[numbers, sort&compress]{natbib} 
\newtheorem{theorem}{Theorem}[section]
\newtheorem{corollary}{Corollary}

\newtheorem{lemma}[theorem]{Lemma}

\theoremstyle{definition}
\newtheorem{definition}[theorem]{Definition}

\def\l{\ell}
\DeclareMathOperator{\Ind}{\mathbf{1}}
\def\ZZ{\mathbb{Z}}
\DeclareMathOperator*{\esssup}{ess~sup}
\newcommand{\D}[2]{\frac{\partial{} #1}{\partial{} #2}}
\newcommand\lnorm[1]{\left\vert{} #1 \right\vert}
\newcommand\norm[1]{\left\Vert{} #1 \right\Vert}
\newcommand\bk[1]{\left\{ #1 \right\}}
\newcommand\subPlus[1]{\left(#1\right)_{+}}
\newcommand{\controlVarsSup}[1]{s^{#1}, g^{#1}, f^{#1}}
\newcommand{\controlVarsSub}[1]{s_{#1}, g_{#1}, f_{#1}}

\def\J{\mathcal{J}}
\title[On the State Constrained Optimal Control\ldots] 
      {On the State Constrained Optimal Control of the Stefan Type Free Boundary Problems}

\author{Ugur G. Abdulla, Evan Cosgrove, Curtis Earl, and Jonathan Goldfarb}

\subjclass{Primary: 35R30, 35R35, 35K20}
 \keywords{inverse Stefan problem, optimal control of parabolic PDE, parabolic free boundary problem, Frechet differentiability, Besov spaces, embedding theorems, trace embeddings.}

 \email{abdulla@fit.edu}
 \email{jgoldfar@fit.edu}
 \email{ecosgrove2011@my.fit.edu}
 \email{cearl2013@my.fit.edu}

\thanks{This research was supported by NSF grant \#1359074}

\thanks{$^*$ Corresponding author: Ugur G. Abdulla}

\begin{document}
\begin{abstract}
We analyze the state constrained inverse Stefan type parabolic free boundary problem as an optimal control problem in the Sobolev-Besov spaces framework.
Boundary heat flux, density of heat sources, and free boundary are components of the control vector.
Cost functional is the sum of the $L_2$-norm declinations of the temperature measurement at the final moment, the phase transition temperature, the final position of the free boundary, and the penalty term, taking into account the state constraint on the temperature.
We prove the existence of optimal control, Frechet differentiability, and optimality condition in the Besov spaces under minimal regularity assumptions on the data.
We pursue space-time discretization through finite differences and prove that the sequence of discrete optimal control problems converges to the original problem both with respect to functional and control.
\end{abstract}

\maketitle


\section{Description of Main Results}
\subsection{Introduction and Motivation}
Consider one-phase Inverse Stefan Problem (ISP) for the second order parabolic PDE:\@
\begin{gather}
Lu := {(a u_x)}_x +b u_x +c u-u_t=f,~ \text{in}~\Omega\label{eq:intro-pde}
\\
u(x,0)= \phi (x),~ 0 \leq x \leq s(0)=s_0\label{eq:intro-initial}
\\
a(0,t) u_x (0,t) = g(t),~ 0 < t < T\label{eq:intro-flux-left}
\\
a(s(t),t) u_x (s(t),t)+\gamma (s(t),t) s'(t) = \chi (s(t),t),~ 0 < t < T\label{eq:intro-stefan-cond}
\\
u(s(t),t)= \mu(t),~ 0 < t < T\label{eq:intro-phase-transition-temp}
\\
u(x,T)=w(x),~ 0 < x < s(T)=:s_{*}\label{eq:intro-measurement-final}
\end{gather}
where $a$, $b$, $c$, $\phi$, $\gamma$, $\chi$, and $w$ are known functions, and
\begin{gather}
  a(x,t) \geq a_0 > 0,\quad
  s_0 > 0\label{eq:a0}
  \\
  \Omega = \bk{(x,t): 0<x<s(t),~0 < t \leq T}.\nonumber
\end{gather}

Solving this problem involves finding both the temperature distribution, $u(x,t)$, the free boundary, $s(t)$, the density of heat sources, $f(x,t)$, and the heat flux on the fixed boundary, $g(t)$.
The functions $a$, $b$, and $c$ are the diffusion, convection, and reaction coefficients in the interior of $\Omega$, respectfully.
The function $\phi$ represents the initial temperature of the system,
and the temperature along the free boundary is denoted by the function $\mu(t)$.
The Stefan condition~\eqref{eq:intro-stefan-cond} ensures conservation of energy across the free boundary $x=s(t)$; that is, the free boundary may only move through the release or absorption of latent heat.
The function $\gamma$ is the coefficient of latent heat of phase change, and $\chi$ represents any additional boundary heat sources.
The function $w(x)$ is the measured temperature at the final moment $t=T$, and $s_{*}$ is the measured location of the free boundary at $t=T$.

Finally, suppose that the temperature is required to remain below a given constant $u_*$.
Then $u$ must also satisfy the constraint
\begin{equation}
    u(x,t) \leq u_{*},~(x,t) \in \Omega.\label{eq:intro-state-constraint}
\end{equation}

Under these conditions, we are required to solve an \textbf{State Constrained Inverse Stefan Problem (SCISP)}: find functions $u(x,t)$ and $s(t)$, the boundary heat flux $g(t)$, and the density of sources $f(x,t)$ satisfying conditions~\eqref{eq:intro-pde}--\eqref{eq:intro-state-constraint}.

Motivation for this type of inverse problem arose from the modeling of bioengineering problems on the laser ablation of biological tissues through a Stefan problem~\eqref{eq:intro-pde}--\eqref{eq:intro-measurement-final}, where $s(t)$ is the ablation depth at the moment $t$.
Lab experiments allow the measurement of the final temperature distribution $w$ and ablation depth $s_{*}$, but unknown parameters of the model such as $g$ and $f$ are very difficult to measure through experiments and thus must be found along with the temperature $u$ and $s$.
In this context, the constraint~\eqref{eq:intro-state-constraint} corresponds to the requirement that no living tissue be raised above a certain damaging temperature $u_{*}$.
Our approach allows the regularization of error contained in final moment temperature measurement and final ablation depth $s_{*}$, as well as any error contained in the phase transition temperature measurement $\mu(t)$.

Research into inverse Stefan problems proceeded in two directions: inverse Stefan problems with given phase boundaries
in~\cite{bell81,budak72,budak73,budak74,cannon64,carasso82,ewing79,ewing79a,goldman97,hoffman81,sherman71}, or inverse problems with unknown phase boundaries in~\cite{baumeister80,fasano77,goldman97,hoffman82,hoffman86,jochum80a,jochum80,knabner83,lurye75,niezgodka79,nochetto87,primicerio82,sagues82,vasilev69,talenti82,yurii80}.

In~\cite{abdulla13,abdulla15}, a new variational formulation of ISP was
formulated and existence of a solution as well as convergence of the method of finite differences was proven.
Fr\'echet differentiability of ISP in the variational formulation was proven in~\cite{abdulla17,abdulla16}.
In this work, we extend the new variational formulation and Frechet differentiability results to the SCISP.\@

The structure of the remainder of the paper is as follows: in Section~\ref{sec:notation} we define all
the functional spaces.
Section~\ref{sec:optimal-control-formulation} formulates the optimal control
problem; the discrete optimal control problem is formulated in
Section~\ref{sec:disc-optimal-control-formulation}.
The main results are formulated in Section~\ref{sec:formulate-main-results}.
In Section~\ref{sec:prelims} preliminary results are proven.
The proofs of the main results are elaborated in Section~\ref{sec:proof-main-results}.
In Section~\ref{thm:existence-convergence} existence of the optimal control and convergence of the sequence of discrete optimal control problems to original problem is proved.
Frechet differentiability and the form of the Frechet differential are
established in Section~\ref{sec:frechet-differentiability}.
\subsection{Notation}\label{sec:notation}
\begin{itemize}
\item Define
  \(
    \subPlus{u} := \max(u;0)
  \).
  Let $U \subset \Re^n$ be a domain and define $Q_T = (0,1)\times (0,T)$.
  We also make use of the notion of weak differentiability and the spaces of Sobolev functions~\cite{besov79,ladyzhenskaya68,nikolskii75,solonnikov64,solonnikov65}:
  \item The Sobolev space $W_{p}^{\l}(0,T)$, for $\l = 1, 2, \ldots$ and $p>1$, is the Banach space of $L_{p}(0,T)$ functions whose weak derivatives up to order $\l$ exist and are in $L_{p}(0,T)$.
  The norm in $W_{p}^{\l}(0,T)$ is
  \[
    \norm{u}_{W_p^{\l}(0,T)}^p
    := \sum_{k=0}^{\l} \norm{\frac{d^k u}{d t^k}}_{L_p(0,T)}^p.
  \]
  \item For $\l > 0$, the Sobolev-Besov space $B_p^{\l}(U)$ is the Banach space of measurable functions with finite norm
  \[
    \norm{u}_{B_p^{\l}(U)}
    :=
    \norm{u}_{W_p^{[\l]}(U)}
    + [u]_{B_p^{\l}(U)}.
  \]
  If \(\l \not\in \ZZ_+\), then the seminorm is given by
  \[
    [u]_{B_p^\l(U)}^p
    :=\int_{U} \int_{U} \frac{
    \lnorm{
      \D{^{[\l]} u(x)}{x^{[\l]}}
      - \D{^{[\l]} u(y)}{x^{[\l]}}
    }^p
    }{
    \lnorm{x-y}^{1 + p(\l - [\l])}
    } \,dx \,dy,
  \]
  while if $\l \in \ZZ_+$, the seminorm $[u]_{B_p^{\l}(U)}$ is given by
  \[
    [u]_{B_p^\l(U)}^p
    :=
    \int_{U} \int_{U} \frac{\lnorm{
        \D{^{\l-1} u(x)}{x^{\l-1}}
        - 2 \D{^{\l-1} u\left(\frac{x+y}{2}\right)}{x^{\l-1}}
        + \D{^{\l-1} u(y)}{x^{\l-1}}
      }^p}{\lnorm{x-y}^{1+p}}
    \,dy \,dx.\nonumber
  \]
  By~\cite[\S 18, Thm.\ 9]{besov79a}, it follows that for $p=2$ and $\l \in
    \ZZ_+$, the $B_p^{\l}(U)$ norm is equivalent to the $W_p^{\l}(U)$ norm
  (i.e.\ the two spaces coincide.)
  In this work, we will use the notation $B_p^{\l}$ for Sobolev-Besov
  functions with any $\l>0$.

  \item Let $1 \leq p < \infty$, $0<\l_1,\l_2$.
  The Besov space $B_{p,x,t}^{\l_1, \l_2}(Q_T)$ is defined as the closure of the set of smooth functions under the norm
  \begin{gather*}
    \norm{u}_{B_{p,x,t}^{\l_1, \l_2}(Q_T)}
    = \left(\int_0^T \norm{u(x,t)}_{B_p^{\l_1}(0,1)}^p \,dt\right)^{1/p}
    \\
    + \left(\int_0^1 \norm{u(x,t)}_{B_p^{\l_2}(0,T)}^p \,dx\right)^{1/p}.
  \end{gather*}
  When $p=2$, if either $\l_1$ or $\l_2$ is an integer, the Besov seminorm may
  be replaced with the corresponding Sobolev seminorm due to equivalence of the
  norms.

  \item $V_{2}(\Omega)$ is the subspace of $B_{2}^{1,0}(\Omega)$ for which the norm
  \[
    \norm{u}_{V_{2}(\Omega)}^2
    = \esssup_{0\leq t \leq T} \norm{u(\cdot, t)}_{
      L_{2}\big(0,s(t)\big)}^{2}
    + \norm{\D{u}{x}}_{L_{2}(\Omega)}^{2}
  \]
  is finite
  \item $V_{2}^{1,0}(\Omega)$ is a Banach space with norm
  \[
    \norm{u}_{V_{2}^{1,0}(\Omega)}^2
    = \max_{0\leq t \leq T} \norm{u(\cdot, t)}_{L_{2}\big(0,s(t)\big)}^{2}
    + \norm{\D{u}{x}}_{L_{2}(\Omega)}^{2}.
  \]
$V_{2}^{1,0}(\Omega)$ is the completion of $B_{2}^{1,1}(\Omega)$ in the $V_{2}(\Omega)$ norm.
\end{itemize}

\begingroup
\def\scontrolspace{B_2^2}
\def\sdiscretecontrolspace{b_2^2}
\def\gcontrolspace{B_2^1}
\def\gdiscretecontrolspace{b_2^1}
\def\fcontrolspace{L_2}
\def\fdiscretecontrolspace{{\l}_{2}}

\def\controlVars{s, g, f}
\def\controlVarsWithN{\controlVarsSup{n}}
\def\controlVarsWithSupStar{\controlVarsSup{*}}
\def\controlVarsWithDelta{{\Delta{} s}, {\Delta{} g}, {\Delta{} f}}
\def\controlVarsWithdelta{{\delta{} s}, {\delta{} g}, {\delta{} f}}
\def\controlVarsWithTilde{\tilde{s}, \tilde{g}, \tilde{f}}
\def\controlVarsWithBar{\bar{s}, \bar{g}, \bar{f}}
\def\controlVarsWithNSub{\controlVarsSub{n}}
\def\controlVarsWithStar{\controlVarsSub{*}}
\def\controlVarsWithEpsilon{\controlVarsSub{\epsilon}}

\def\controlSpace{H}
\def\controlSpaceFull{\scontrolspace(0,T)\times{} \gcontrolspace(0,T)\times{}
  \fcontrolspace(D)}

\def\controlSpaceWeaklyConverge{\scontrolspace(0,T) \times{} \gcontrolspace(0,T) \times{} \fcontrolspace(D)}

\def\controlVarsStronglyConverge{s, g}
\def\controlVarsStronglyConvergeWithTilde{\tilde{s}, \tilde{g}}
\def\controlVarsStronglyConvergeWithSupStar{s^*, g^*}
\def\controlVarsStronglyConvergeWithN{s^n, g^n}
\def\controlSpaceStronglyConverge{B_2^1(0,T) \times{} L_2(0,T)}

\def\discreteControlVars{[s]_n, [g]_n, [f]_{nN}}
\def\discreteControlVarsList{$[s]_n$, $[g]_n$, $[f]_{nN}$}
\def\discreteControlSet{V_R^n}

\def\P{\mathcal{P}}
\def\Q{\mathcal{Q}}
\def\probIn{\mathcal{I}_n}
\def\fIn{I_n}
\subsection{Optimal Control Problem}\label{sec:optimal-control-formulation}
Fix a sequence of real numbers $A_k\uparrow \infty$.
For each $k=1,2,\ldots$, we wish to minimize the cost functional
\begin{gather}\label{eq:cost-func}
  \mathcal{J}(v)= J(v) + P_{k}(v) \to \inf,
  \\
  J(v) := \beta_0 \int_0^{s(T)}\left|u(x,T;v)-w(x)\right|^2\,dx
  + \beta_1 \int_0^T\left|u(s(t),t;v)-\mu(t)\right|^2\,dt\nonumber
  \\
  + \beta_2 \left|s(T)-s_*\right|^2,\qquad
P_{k}(v) := A_{k}\int_0^T\int_0^{s(t)} \subPlus{u(x,t;v)-u_{*}}^2\, dx \, dt\nonumber
\end{gather}
on the control set $V_R$ defined as
\begin{gather}\label{eq:control-set}
  V_R=\{ v=(\controlVars{}) \in \controlSpace:
  ~ 0< \delta \leq s(t),
  ~ s(0) = s_{0},
  ~ s'(0)=0,
  ~ \norm{v}_H \leq R \},
  \\
  \controlSpace := \controlSpaceFull,\nonumber
  \\
  \norm{v}_{\controlSpace} := \max(
  \norm{s}_{\scontrolspace(0,T)},
  \norm{g}_{\gcontrolspace(0,T)},
  \norm{f}_{\fcontrolspace(D)}
  )\nonumber
\end{gather}
where $D$ is defined by
\[
  D := \bk{(x,t) : 0\leq x\leq \l,~ 0\leq t\leq T},
\]
and $\l=\l(R)>0$ is chosen such that for any control $v\in V_R$, its component
$s$ satisfies $s(t)\leq \l$.
Existence of appropriate $\l$ follows from Morrey's
inequality~\cite{besov79,ladyzhenskaya68}.

For a given control \(v\in \controlSpace{}\), the state vector $u=u(x,t;v)$ solves~\eqref{eq:intro-pde}--\eqref{eq:intro-stefan-cond}.

The cost functional effectively deals with possible measurement error, as it absorbs conditions~\eqref{eq:intro-phase-transition-temp} and~\eqref{eq:intro-measurement-final}.
Furthermore, note that the final term of the cost functional, the penalty term, absorbs~\eqref{eq:intro-state-constraint} into the cost functional.
This term ensures the cost functional will be large if condition~\eqref{eq:intro-state-constraint} is not satisfied.
We seek to prove the existence of a solution to the Optimal Control Problem, as well as convergence of finite difference approximations and differentiability of the new variational formulation.
The energy estimates derived in~\cite{abdulla15} will be a key tool in the proof of the existence and convergence results.

Given control vector \(v\) solution $u(x,t;v)$ of the problem~\eqref{eq:intro-pde}--\eqref{eq:intro-stefan-cond} will be understood in the following sense.

\begin{definition}\label{w211-soln-defn}
The function $u \in B_{2}^{1,1}(\Omega)$ is called a weak solution of~\eqref{eq:intro-pde}--\eqref{eq:intro-stefan-cond} if $u(x,0)=\phi(x) \in B_2^1(0,s_0)$ and
\begin{gather}
0=\int_{0}^{T}\int_{0}^{s(t)}[ a u_{x}\Psi_{x}-bu_{x}\Psi - c u \Psi + u_{t} \Psi+f\Psi] \,dx\,dt \nonumber
\\
 +\int_{0}^{T}[ \gamma(s(t),t)s'(t)-\chi(s(t),t)]\Psi(s(t),t)\, dt
+\int_{0}^{T}g(t)\Psi(0,t)\, dt\label{eq:w211-soln-defn}
\end{gather}
for arbitrary $\Psi \in B_{2}^{1,1}(\Omega)$.
\end{definition}
We also need the notion of the weak solution in $V_2(\Omega)$.
\begin{definition}\label{v2-soln-defn}
We say $u \in V_{2}(\Omega)$ is a weak solution of~\eqref{eq:intro-pde}--\eqref{eq:intro-stefan-cond} if the following integral identity is satisfied
\begin{gather}
\int_0^T \int_0^{s(t)}[-au_x\Psi_x+bu_x\Psi+cu\Psi+u\Psi_t-f\Psi]\,dx\,dt \nonumber
\\
+ \int_0^T [-\gamma(s(t),t)s'(t)+\chi(s(t),t)+u(s(t),t)s'(t)]\Psi(s(t),t)\,dt \nonumber
\\
- \int_0^T g(t)\Psi(0,t)\,dt + \int_0^{s(0)}\phi(x)\Psi(x,0)\,dx = 0 \label{eq:v2-soln-defn}
\end{gather}
for arbitrary $\Psi \in B_2^{1,1}(\Omega)$ with $\Psi |_{t=T}=0$.
\end{definition}

If $u$ is a weak solution in $V_2^{1,0}(\Omega)$ or $B_2^{1,1}(\Omega)$, the traces $u|_{x=0}, u|_{x=s(t)}$ and $u|_{t=T}$ exists in $L_2(0,T)$ and $L_2(0,s(T))$ respectively, if $s\in B_2^2(0,T)$ (\cite{ladyzhenskaya68}), and therefore functional $\mathcal{J}(v)$ is well defined for $v\in V_R$.

\subsection{Discrete Optimal Control Problem}\label{sec:disc-optimal-control-formulation}
Let
\[\omega_{\tau}=\bk{ t_{k}=k  \tau,~k=0,1,\ldots,n} \]
be a grid on $[0,T]$ with $\tau=\frac{T}{n}$.
Let us now introduce a spatial grid.
Given a vector $[s]_n \in \Re^{n+1}$, let $(p_0,p_1,\ldots,p_n)$ be a permutation of $(0,1,\ldots,n)$ such that
\[
  s_{p_0}\leq s_{p_1}\leq \cdots \leq s_{p_n}.
\]
In particular, according to this permutation for arbitrary $k$ there exists a unique $j_k$ such that
\[
  s_k=s_{p_{j_k}}.
\]
Furthermore, unless it is necessary in context, we will write $j$ instead of the subscript $j_k$.
Let
\[
  \omega_{p_0}=\bk{
    x_{i}:
    x_i=i h,
    ~i=0,1,\ldots,m_0^{(n)}
  }
\]
be a grid on $[0,s_{p_0}]$ and $h=\frac{s_{p_0}}{m_0^{(n)}}$.
Furthermore, we will assume that
\begin{equation}
  h = O(\sqrt{\tau}) \quad \text{as}~ \tau \rightarrow 0.\label{eq:htau}
\end{equation}
We continue construction of the spatial grid by induction.
Having constructed $\omega_{p_{k-1}}$ on $[0,s_{p_{k-1}}]$, we construct
\[ \omega_{p_k}=\{ x_i:~i=0,1,\ldots, m_k^{(n)} \} \]
on $[0,s_{p_{k}}]$, where $m_k^{(n)}\geq m_{k-1}^{(n)}$, and this inequality is strict if and only if $s_{p_{k}}>s_{p_{k-1}}$; for $i\leq m_{k-1}^{(n)}$ the points $x_i$ are the same as in the grid $\omega_{p_{k-1}}$.
Finally, if $s_{p_{n}}<\l$, then we introduce a grid on $[s_{p_n},l]$,
\[
  \overline{\omega}=\{x_i: x_i=s_{p_n}+(i-m_n^{(n)}) \overline{h},
  ~i=m_n^{(n)},\ldots, N \}
\]
of stepsize order $h$; i.e., $\overline{h}=O(h)$ as $h \rightarrow 0$.
Furthermore we simplify the notation and write $m_k^{(n)}\equiv m_k$.
Let
\[ h_i=x_{i+1}-x_i, ~i=0,1,\ldots,N-1; \quad \Delta := \max_{i=0,1,\ldots,N-1} h_i\]
and assume that
\(
m_k \rightarrow +\infty
\)
as
\(
n\rightarrow \infty
\).
Define the discrete control set
\begin{gather}
  \discreteControlSet{} = \Big{\{}
  [v]_n = (\discreteControlVars{}) \in \bar{H}:~
  0< \delta \leq s_k;~
  \norm{[v]_n}_{\bar{H}} \leq R
  \Big{\}}\label{eq:discrete-control-set}
\end{gather}
where
\begin{gather*}
  \bar{H}:=\Re^{n + 1} \times \Re^{n + 1} \times \Re^{nN}
  \\
   [s]_n=(s_i, i=\overline{0,n}; s_1=s_0),
   ~[g]_n=(g_i, i=\overline{0,n}),
   \\
   [f]_{nN}=(f_{ik},~ i=\overline{0,N-1},~ k=\overline{1,n})
  \\
  \norm{[v]_n}_{\bar{H}} := \max \left(
  \norm{[s]_{n}}_{\sdiscretecontrolspace};
  \norm{[g]_{n}}_{\gdiscretecontrolspace};
  \norm{[f]_{nN}}_{\fdiscretecontrolspace}
  \right),
  \\
  \norm{[g]_{n}}_{\gdiscretecontrolspace}^{2}
  = \sum_{k=0}^{n-1}\tau g_{k}^{2}
  + \sum_{k=1}^{n}\tau g_{k,\bar{t}}^{2},
  \quad
  \norm{[s]_{n}}_{\sdiscretecontrolspace}^{2}
  = \norm{[s]_{n}}_{\gdiscretecontrolspace}^2
  + \sum_{k=1}^{n-1}\tau s_{k,\bar{t}t}^{2},
  \\
  \norm{[f]_{nN}}_{\fdiscretecontrolspace{}}^2
  = \sum_{k=1}^n \sum_{i=0}^{N-1} \tau h_i f_{ik}^2,
\end{gather*}
and the standard notation for the finite differences is used,
\begin{gather*}
s_{k,\bar{t}} = \frac{s_k-s_{k-1}}{\tau}, \quad
s_{k,\bar{t}t}=\frac{s_{k+1}-2s_k+s_{k-1}}{\tau^2}.
\end{gather*}

Introduce the two mappings $\Q_n$ and $\P_n$ between the continuous and discrete control sets $V_R$ and $\discreteControlSet{}$, where

\[
  \Q_{n}(v)=[v]_{n}=(\discreteControlVars{}) \quad \text{for}~ v\in V_R,
\]
where
\begin{gather*}
  s_k=s(t_k),~k=\overline{2,n}, \quad g_k=g(t_k),~k=\overline{0,n},
  \\
  f_{ik}=\frac{1}{\tau h_i} \int_{t_{k-1}}^{t_k}\int_{x_i}^{x_{i+1}}f(x,t) \,dx\,dt,~i=\overline{0,N-1},~k=\overline{1,n},
\end{gather*}
and
\[
  \P_{n}([v]_{n})=v^{n}=(\controlVarsWithN{})\in \controlSpaceFull{} \quad \text{for}~ [v]_{n} \in \discreteControlSet{},
\]
where
\begin{gather*}
s^n(t)=
\begin{cases}
s_0&~ 0\leq t \leq \tau,\\
s_{k-1}+(t-t_{k-1}-\frac{\tau}{2})s_{k-1,\bar{t}}+\frac{1}{2}(t-t_{k-1})^2 s_{k-1,\bar{t}t}& t_{k-1}\leq t \leq t_k,
\end{cases}
\\
g^n(t)=g_{k-1} + g_{k,\bar{t}}(t-t_{k-1}),\quad t_{k-1} \leq t \leq t_k, k=\overline{1,n},
\intertext{and}
f^{n}(x,t) = f_{ik},\quad x_i\leq x<x_{i+1},\quad t_{k-1}<t\leq t_{k},\quad i =\overline {0,N-1}, j = \overline{1,n}.
\end{gather*}
Introduce the following notation for Steklov averages:
\begin{gather*}
r_{k}=\frac{1}{\tau}\int_{t_{k-1}}^{t_{k}}r(t)\,dt,\quad
d_{ik}=\frac{1}{h_i \tau} \int_{x_i}^{x_{i+1}}\int_{t_{k-1}}^{t_k}
d(x,t)\,dt\,dx,
\\
w_i=\frac{1}{h_i}\int_{x_i}^{x_{i+1}}w(x)\,dx
\end{gather*}
where \(i=\overline{0,N-1}\), \(k=\overline{1,n}\), \(d \in \bk{a,b,c,f,f^{n}}\),
\(r \in \bk{\mu,g,g^n}\).
Given $v=(\controlVars{}) \in V_R$, we define the Steklov averages of traces by
\begin{equation}
\chi^{k}_s=\frac{1}{\tau} \int_{t_{k-1}}^{t_{k}}\chi(s(t),t) \,dt, \
(\gamma_s s')^k=\frac{1}{\tau} \int_{t_{k-1}}^{t_{k}}\gamma(s(t),t)s'(t) \,dt,\quad k=\overline{1,n}.\label{eq:stek-avg-trace}
\end{equation}
Given $[v]_{n}=(\discreteControlVars{}) \in \discreteControlSet{}$ we define Steklov averages $\chi^{k}_{s^n}$ and $(\gamma_{s^n} (s^n)')^k$ through~\eqref{eq:stek-avg-trace} with $s$ replaced by $s^n$.

We now introduce the notion of a solution in the discrete sense through a discretization of the integral identity~\eqref{eq:w211-soln-defn}.

\begin{definition}\label{discretestatevector}
Given a discrete control vector $[v]_{n}$, the vector function
\[ \big[u([v]_{n})\big]_{n}=\big(u(0),u(1),\ldots,u(n)\big), ~u(k) = \big(u_0(k), \ldots, u_N(k)\big)\in \Re^{N+1}, \]
\(k=0,\ldots,n\) is called a discrete state vector if
\begin{description}
\item[a] The first $m_0+1$ components of the vector $u(0)\in \Re^{N+1}$ satisfy
\[ u_i(0)=\phi_i := \phi(x_i), ~i=0,1,\ldots,m_0; \]
\item[b] For arbitrary $k=1,\ldots,n$ first $m_j+1$ components of the vector $u(k)\in \Re^{N+1}$ solve the following system of $m_j+1$ linear algebraic equations:
\begin{gather}
\Big [ a_{0k}+h b_{0k}-h^2c_{0k}+\frac{h^2}{\tau} \Big] u_0(k) - \Big [
a_{0k}+hb_{0k} \Big] u_1(k)=\frac{h^2}{\tau}u_0(k-1)\nonumber
\\
-h^2f_{0k} - hg^n_{k},\nonumber
\\
-a_{i-1,k}h_i u_{i-1}(k)+\Big [ a_{i-1,k}h_i+a_{ik}h_{i-1}+b_{ik}h_i h_{i-1}-c_{ik}h_i^2h_{i-1}+\frac{h_i^2h_{i-1}}{\tau} \Big] u_i(k) \nonumber
\\
-\Big [ a_{ik}h_{i-1}+b_{ik}h_i h_{i-1} \Big] u_{i+1}(k) = -h_i^2h_{i-1}f_{ik}+\frac{h_i^2h_{i-1}}{\tau}u_i(k-1), \nonumber
\intertext{for \(i=1,\ldots,m_j-1\), and}
-a_{m_j-1,k} u_{m_j-1}(k)+a_{m_j-1,k} u_{m_j}(k)=-h_{m_j-1} \Big [ (\gamma_{s^n} (s^n)')^k-\chi^{k}_{s^n} \Big];\label{eq:disc-lin-eq}
\end{gather}
\item[c] For arbitrary $k=0,1,\ldots,n$, the remaining components of $u(k)\in \Re^{N+1}$ are calculated as
\[ u_i(k)= \hat{u}(x_i;k), ~m_j\leq i \leq N \]
where $\hat{u}(x;k) \in B_2^1(0,\l)$ is a piecewise linear interpolation of $\bk{u_i(k)}$, that is to say
\[ \hat{u}(x;k)=u_i(k)+\frac{u_{i+1}(k)-u_i(k)}{h_i} (x-x_i),~x_i\leq x\leq x_{i+1},~ i=0,\ldots,m_j-1, \]
iteratively continued for $x \geq s_k$ as
\begin{equation}
\hat{u}(x;k)=\hat{u}(2^n s_k-x;k), ~2^{n-1}s_k \leq x \leq 2^n s_k,~ n=1,2,\ldots.\label{eq:disc-st-ext}
\end{equation}
\end{description}
\end{definition}
Note that no more than
\(
  n^*=1+\log_2[ \l/\delta]
\)
reflections are required to cover $[0,\l]$.
The system~\eqref{eq:disc-lin-eq} is obtained after discretizing~\eqref{eq:w211-soln-defn} and performing summation by parts on the resulting summation identity
\begin{gather}
\sum_{i=0}^{m_j-1}h_i \Big [ a_{ik}u_{ix}(k)\eta_{ix}-b_{ik}u_{ix}(k)\eta_i-c_{ik}u_i(k)\eta_i+
f_{ik}^{n}\eta_i+u_{i\bar{t}}(k)\eta_i \Big] \nonumber
\\
+ \Big [ (\gamma_{s^n} (s^n)')^k-\chi^{k}_{s^n} \Big]\eta_{m_j}+g^n_k \eta_0=0,\label{eq:disc-int-id}
\end{gather}
where $\eta_i$, $i=0,1,\ldots,m_j$ are arbitrary numbers.

With this definition, we may now consider the discrete version of our optimal control problem, where we wish to minimize the cost functional
\begin{gather}\label{eq:disc-cost-func}
  \probIn([v]_n)= I_n([v]_n) + P_{k}^n([v]_n),
  \\
  I_n([v]_n) := \beta_0\sum_{i=0}^{m_{j_n}-1}h_i|u_i(n)-w_i|^2+\beta_1\sum_{k=1}^n \tau|u_{m_{j_k}}(k)-\mu_k|^2+\beta_2|s_n-s_*|^2,\nonumber
\\
P_{k}^n([v]_n) := A_k\sum_{k=1}^n\sum_{i=0}^{m_{j_k}-1}\tau h_i \subPlus{u_i(k)-u_*}^2,\nonumber
\end{gather}
on the control set $\discreteControlSet{}$, subject to the discrete state vector.
We will call this Problem $\probIn{}$.

Finally, we define piecewise constant and piecewise linear interpolations of the discrete state vector, which we will use later.
Given discrete state vector $[u([v]_{n})]_{n}=(u(0),u(1),\ldots,u(n))$, let
\begin{gather*}
  u^\tau(x,t)=\hat{u}(x;k), \quad \text{if}~ t_{k-1}<t\leq t_k, ~0\leq x \leq \l,
  ~k=\overline{0,n},
  \\
  \hat{u}^\tau(x,t)=\hat{u}(x;k-1)+\hat{u}_{\bar{t}}(x;k)(t-t_{k-1}), \quad
  \text{if}~ t_{k-1}<t\leq t_k, ~0\leq x \leq \l, ~k=\overline{1,n},
  \\
  \hat{u}^\tau(x,t)= \hat{u}(x;n), \quad \text{if}~ t\geq T, ~0\leq x \leq \l,
  \\
\tilde{u}^\tau(x,t)=u_i(k), \quad \text{if}~ t_{k-1}< t \leq t_k, ~x_i \leq x <
x_{i+1}, ~k=\overline{1,n}, ~i=\overline{0,N-1}.
\end{gather*}
Note that $u^\tau \in V_2(D)$, $\hat{u}^\tau \in B_2^{1,1}(D)$, and $\tilde{u}^\tau \in L_2(D)$.
Additionally, we employ standard notation for the difference quotients of the discrete state vector:
\[
  u_{ix}(k)=\frac{u_{i+1}(k)-u_i(k)}{h_i},
  ~u_{i\bar{t}}=\frac{u_i(k)-u_i(k-1)}{\tau}, ~\quad \text{etc.}
\]
Let $\phi^n$ be a piecewise constant approximation to $\phi$:
\[
  \phi^n(x) = \phi_i,~x_i < x \leq x_{i+1},~i=0,\ldots,N-1.
\]
\endgroup
\subsection{Formulation of Main Results}\label{sec:formulate-main-results}

The main results on discretization of the optimal control problem are as follows
\begingroup
\def\scontrolspace{B_2^2}
\def\sdiscretecontrolspace{b_2^2}
\def\gcontrolspace{B_2^1}
\def\gdiscretecontrolspace{b_2^1}
\def\fcontrolspace{L_2}
\def\fdiscretecontrolspace{{\l}_{2}}

\def\controlVars{s, g, f}
\def\controlVarsWithN{\controlVarsSup{n}}
\def\controlVarsWithSupStar{\controlVarsSup{*}}
\def\controlVarsWithDelta{{\Delta{} s}, {\Delta{} g}, {\Delta{} f}}
\def\controlVarsWithdelta{{\delta{} s}, {\delta{} g}, {\delta{} f}}
\def\controlVarsWithTilde{\tilde{s}, \tilde{g}, \tilde{f}}
\def\controlVarsWithBar{\bar{s}, \bar{g}, \bar{f}}
\def\controlVarsWithNSub{\controlVarsSub{n}}
\def\controlVarsWithStar{\controlVarsSub{*}}
\def\controlVarsWithEpsilon{\controlVarsSub{\epsilon}}

\def\controlSpace{H}
\def\controlSpaceFull{\scontrolspace(0,T)\times{} \gcontrolspace(0,T)\times{}
  \fcontrolspace(D)}

\def\controlSpaceWeaklyConverge{\scontrolspace(0,T) \times{} \gcontrolspace(0,T) \times{} \fcontrolspace(D)}

\def\controlVarsStronglyConverge{s, g}
\def\controlVarsStronglyConvergeWithTilde{\tilde{s}, \tilde{g}}
\def\controlVarsStronglyConvergeWithSupStar{s^*, g^*}
\def\controlVarsStronglyConvergeWithN{s^n, g^n}
\def\controlSpaceStronglyConverge{B_2^1(0,T) \times{} L_2(0,T)}

\def\discreteControlVars{[s]_n, [g]_n, [f]_{nN}}
\def\discreteControlVarsList{$[s]_n$, $[g]_n$, $[f]_{nN}$}
\def\discreteControlSet{V_R^n}

\def\P{\mathcal{P}}
\def\Q{\mathcal{Q}}
\def\probIn{\mathcal{I}_n}
\def\fIn{I_n}
\begin{theorem}\label{thm:existence-convergence}
  Assume
  \begin{equation}
  \begin{gathered}
    a,b,c \in L_\infty(D),
    \\
    \phi \in B_2^1(0,s_0), \quad
    \gamma,\chi \in B_2^{1,1}(D), \quad
    \mu, w \in L_2(0,T),
    \\
    \D{a}{x} \in L_{\infty}(D), \quad
    \int_0^T \esssup_{0\leq x \leq \l}\lnorm{\D{a}{t}} \, dt<\infty.
  \end{gathered}\nonumber%
  \end{equation}
  Then
  \begin{enumerate}
  \item The optimal control problem $I$ has a solution.
    That is,
    \[
      V_*=\bk{ v\in V_R: ~\mathcal{J}(v)=\mathcal{J}_* \equiv \inf_{v\in V_R}
        \mathcal{J}(v) } \neq \emptyset.
    \]
  \item
    The sequence of discrete optimal control problems $I_n$ approximates the optimal control problem $I$ with respect to functional, i.e.
    \begin{gather*}
      \lim_{n\to +\infty} \mathcal{I}_{n_*}=\mathcal{J}_*,
    \end{gather*}
    where
    \[
      \mathcal{I}_{n_*}=\inf_{\discreteControlSet{}} \mathcal{I}_n([v]_n), ~n=1,2,\ldots.
    \]
    The sequence of problems \(I_n\) approximates the problem \(I\) with respect to
    control, in that if $[v]_{n_\epsilon}\in \discreteControlSet{}$ is chosen such that
    \[
      \mathcal{I}_{n_*} \leq \mathcal{I}_n([v]_{n_\epsilon})\leq
      \mathcal{I}_{n_*}+\epsilon_n, ~\epsilon_n \downarrow 0,
    \]
    then the sequence $v_n=(\controlVarsWithN{})=\P_n([v]_{n_\epsilon})$
    converges to some element $v_* \in V_*$ weakly
    in $\controlSpaceWeaklyConverge{}$.
    Moreover, $\hat{u}^\tau$ converges to the solution $u(x,t;v_*) \in
    B_2^{1,1}(\Omega_*)$ of the Neumann
    problem~\eqref{eq:intro-pde}--\eqref{eq:intro-stefan-cond} weakly in
    $B_2^{1,1}(\Omega_*')$, where $\Omega_*'$ is any interior subset of $\Omega_*$.
  \end{enumerate}
\end{theorem}
\endgroup

\begingroup
\def\sbar{\overline{s}}
\def\ubar{\overline{u}}
\def\fbar{\overline{f}}
\def\ybar{\bar{y}}
\def\stilde{\widetilde{s}}
\def\utilde{\widetilde{u}}
\def\ftilde{\widetilde{f}}
\def\shat{\widehat{s}}
\def\ustar{u_{*}}
\def\Omegahat{\widehat{\Omega}}
\newcommand\tildebar[1]{\widetilde{\overline{#1}}}
\def\utildebar{\tildebar{u}}
\def\solnspace{B_{2, x, t}^{5/2 + 2\alpha, 5/4 + \alpha}}
\def\densityspace{B_{2, x, t}^{1/2 + 2\alpha, 1/4 + \alpha}}
\def\ivspace{B_2^{3/2 + 2\alpha}}
\def\uxtrspace{B_2^{1 + \alpha}}
\def\uxxtrspace{B_2^{1/2 + \alpha}}
\def\uxxxtrspace{B_2^{\alpha}}
\def\controlVars{s, g, f}
\def\controlVarsWithDelta{{\Delta{} s}, {\Delta{} g}, {\Delta{} f}}
\def\controlVarsWithBar{\bar{s}, \bar{g}, \bar{f}}

\def\scontrolspace{B_2^2}
\def\gcontrolspace{B_2^{1/2 + \alpha}}
\def\fcontrolspace{B_{2, x, t}^{1, 1/4 + \alpha}}

\def\controlSpace{\tilde{H}}
\def\controlSpaceFull{\scontrolspace(0,T)\times{} \gcontrolspace(0,T)\times{}
  \fcontrolspace(D)}

\def\muspace{B_2^{1/4}}
\def\adjointsolnspace{B_{2, x, t}^{2, 1}}
\def\chigammaspace{B_{2, x, t}^{3/2 + 2\alpha^*,3/4 + \alpha^*}}
We now consider the problem of Frechet differentiability of the functional \(\J{}\).
\begin{definition}\label{defn:Frechet}
  Let $V$ be a convex and closed subset of a Banach space $H$.
  We say that the functional $\J:V\to \Re$ is differentiable in the sense of Frechet at the point $v\in V$ if there exists an element $\J'(v) \in H'$ of the dual space such that
  \begin{equation}\label{frechetderivativedefinition}
    \J(v+h)-\J(v)=\left\langle{}\J'(v),{h}\right\rangle_H+o(h,v),
  \end{equation}
  where $v+h\in V\cap\bk{u: \norm{u}< \gamma}$ for some $\gamma>0$; $\langle{} \cdot,\cdot\rangle{}_H$ is a pairing
  between $H$ and its dual $H'$, and
  \[
    \frac{o(h,v)}{\norm{h}} \to 0,\quad \text{as}~\norm{h} \to 0.
  \]
  The expression $dJ(v)=\left\langle{}\J'(v),{\cdot}\right\rangle_H$ is called a Frechet differential of $\J$ at $v\in V$, and the element $\J'(v) \in H'$ is called Frechet derivative or gradient of $\J$ at $v\in V$.
\end{definition}

\begin{definition}\label{defn:adjoint}
  For a given control vector $v$ and corresponding state vector $u = u(x, t; v)$, $\psi \in B_2^{2,1}(\Omega)$ is a solution to the adjoint problem if
  \begin{equation}
  \begin{gathered}
    L^* \psi := {\big(a\psi_x\big)}_x - {(b\psi)}_x + c\psi + \psi_t = -2 A_k
    \subPlus{u(x,t) - u_*},\quad\text{in}~\Omega
    \\
    \psi(x, T) = 2\beta_0(u(x, T) - w(x)),~0 \leq x \leq s(T)
    \\
    a(0, t)\psi_x(0, t) - b(0, t)\psi(0, t)=0,~0 \leq t \leq T
    \\
    {\Big[a\psi_x - \big(b + s'(t)\big)\psi\Big]}_{x=s(t)} = 2\beta_1(u(s(t), t) - \mu(t)), ~0 \leq t \leq T
  \end{gathered}\label{eq:adjoint-problem}
  \end{equation}
\end{definition}

The main result on Frechet differentiability reads:
\begin{theorem}[Frechet Differentiability]\label{thm:gradient}
  Let \(0< \alpha < \alpha^* \ll 1\) and
  \begin{equation}
    \begin{gathered}
      a,a_x, b,c \in C_{x,t}^{\frac{1}{2}+2\alpha^*,\frac{1}{4}+\alpha^*}(D),
      \\
      \phi \in B_2^{\frac{3}{2}+2\alpha}(0,s_0),\quad
      \chi,\gamma \in B_{2,x,t}^{\frac{3}{2}+2\alpha^*,\frac{3}{4}+\alpha^*}(D),
      \\
      w \in B_2^1(0,\l),\quad
      \mu \in B_2^{\frac{1}{4}}(0,T),
    \end{gathered}\label{eq:frechet-result-datacond}
  \end{equation}
  where the control vector \(v\) belongs to the control set:
  \begin{gather}\label{eq:control-set2}
    W_R=\{ v=(\controlVars{}) \in \controlSpace{}:
    ~ 0< \delta \leq s(t),
    ~ s(0)=s_0,
    ~ g(0)=a(0,0)\phi'(0),
    \\
    ~ \chi(s_0,0)=\phi'(s_0)a(s_0)+\gamma(s_0,0)s'(0),
    ~ \norm{v}_{\controlSpace{}} \leq R \}, \nonumber
    \\
    \controlSpace{} := \controlSpaceFull{},\nonumber
    \\
    \norm{v}_{\controlSpace{}} := \max\bigg(
    \norm{s}_{\scontrolspace(0,T)},
    \norm{g}_{\gcontrolspace(0,T)},
    \norm{f}_{\fcontrolspace(D)}
    \bigg).\nonumber
  \end{gather}
  The functional $\J(v)$ is differentiable in the sense of Frechet on \(W_R\), and the Frechet differential is
  \begin{gather}
    \left\langle{}\J'(v),{\Delta v}\right\rangle_{\controlSpace{}}
    = -\int_{\Omega} \psi {\Delta f} \,dx\,dt
    - \int_{0}^T \psi(0,t) {\Delta g}(t)\,dt
    -\int_0^T \big[\gamma \psi\big]_{x=s(t)}{\Delta s}'(t)\,dt\nonumber
    \\
    + \int_0^T \left[2 \beta_1 (u-\mu) u_x
      + \psi \left(\chi_x - \gamma_x s' -\big(a u_x\big)_x\right)
      + A_k \lnorm{\subPlus{u - u_*}}^2 \right]_{x=s(t)}{\Delta s}(t)\, dt  \nonumber
    \\
    + \left(\beta_0\lnorm{u(s(T),T) - w(s(T))}^2 + 2\beta_2(s(T)-s_*)\right){\Delta s}(T),\label{eq:gradient-full}
  \end{gather}
  where $\J'(v)\in \controlSpace{}'$ is the Frechet derivative, $\psi$ is a solution to the adjoint problem in the sense of Definition~\ref{defn:adjoint}, and ${\Delta v} = (\controlVarsWithDelta{})$ is a variation of the control vector $v \in W_R$ such that $v + {\Delta v} \in W_R$.
\end{theorem}
\begin{corollary}[Optimality Condition]\label{optimalitycondition}
  If $\mathbf{v}$ is an optimal control,
  then the following variational inequality is satisfied:
  \begin{equation}
    \left\langle \J'(\mathbf{v}), v-\mathbf{v} \right\rangle_{\controlSpace{}}\geq 0\label{eq:optimality-condition}
  \end{equation}
  for arbitrary $v \in W_R$.
\end{corollary}
\endgroup
\section{Preliminary Results}\label{sec:prelims}

\begingroup
\def\scontrolspace{B_2^2}
\def\sdiscretecontrolspace{b_2^2}
\def\gcontrolspace{B_2^1}
\def\gdiscretecontrolspace{b_2^1}
\def\fcontrolspace{L_2}
\def\fdiscretecontrolspace{{\l}_{2}}

\def\controlVars{s, g, f}
\def\controlVarsWithN{\controlVarsSup{n}}
\def\controlVarsWithSupStar{\controlVarsSup{*}}
\def\controlVarsWithDelta{{\Delta{} s}, {\Delta{} g}, {\Delta{} f}}
\def\controlVarsWithdelta{{\delta{} s}, {\delta{} g}, {\delta{} f}}
\def\controlVarsWithTilde{\tilde{s}, \tilde{g}, \tilde{f}}
\def\controlVarsWithBar{\bar{s}, \bar{g}, \bar{f}}
\def\controlVarsWithNSub{\controlVarsSub{n}}
\def\controlVarsWithStar{\controlVarsSub{*}}
\def\controlVarsWithEpsilon{\controlVarsSub{\epsilon}}

\def\controlSpace{H}
\def\controlSpaceFull{\scontrolspace(0,T)\times{} \gcontrolspace(0,T)\times{}
  \fcontrolspace(D)}

\def\controlSpaceWeaklyConverge{\scontrolspace(0,T) \times{} \gcontrolspace(0,T) \times{} \fcontrolspace(D)}

\def\controlVarsStronglyConverge{s, g}
\def\controlVarsStronglyConvergeWithTilde{\tilde{s}, \tilde{g}}
\def\controlVarsStronglyConvergeWithSupStar{s^*, g^*}
\def\controlVarsStronglyConvergeWithN{s^n, g^n}
\def\controlSpaceStronglyConverge{B_2^1(0,T) \times{} L_2(0,T)}

\def\discreteControlVars{[s]_n, [g]_n, [f]_{nN}}
\def\discreteControlVarsList{$[s]_n$, $[g]_n$, $[f]_{nN}$}
\def\discreteControlSet{V_R^n}

\def\P{\mathcal{P}}
\def\Q{\mathcal{Q}}
\def\probIn{\mathcal{I}_n}
\def\fIn{I_n}
\begin{lemma}\label{lem:trunc-props}\cite{evans92}
  For any function $u \in B_p^1(U)$, the function
  \begin{equation}
    u^+ \in B_p^1(U)
    ~\text{and}~
    \norm{u^+}_{B_p^1(U)}
    \leq \norm{u}_{B_p^1(U)}.\label{eq:trunc-H1}
  \end{equation}
  Furthermore, the mapping $u \to u^+$ is a contraction in $L_p$, in that
  \begin{equation}
    \norm{u^+ - v^+}_{L_p(U)}
    \leq \norm{u - v}_{L_p(U)}\label{eq:trunc-contractive}
  \end{equation}
\end{lemma}
Proof of the existence and uniqueness of the discrete state vector $\big[u([v]_n)\big]_n$ coincides with the proof
given in~\cite{abdulla15}:
\begin{lemma}\label{existencediscretestate}\cite{abdulla15}
For sufficiently small time step $\tau$, there exists a unique discrete state vector $[u([v]_{n})]_{n}$ for arbitrary discrete control vector $[v]_{n} \in \discreteControlSet{}$.
\end{lemma}
The following is the general criterion for the convergence of the sequence of discrete optimal control problems to the
continuous optimal control problem.
\begin{lemma}\label{lem:Vasil'ev}\cite{vasilev81}
A sequence of discrete optimal control problems $I_{n}$ approximates a
continuous optimal control problem $I$ if and only if the following conditions
are satisfied:
\begin{description}
\item[i] for arbitrary sufficiently small $\epsilon>0$ there exists number $M_1=M_1(\epsilon)$ such that $\mathcal{Q}_M(v)\in V^{M}_R$ for all $v \in V_{R-\epsilon}$ and $M\geq M_1$; and for any fixed $\epsilon>0$ and for all $v\in V_{R-\epsilon}$ the following inequality is satisfied:
\[ \limsup_{M\to \infty} \Big ( \mathcal{I}_M(\mathcal{Q}_M(v))-\mathcal{J}(v) \Big ) \leq 0.
\]
\item[ii] for arbitrary sufficiently small $\epsilon>0$ there exists number $M_2=M_2(\epsilon)$ such that $\mathcal{P}_M([v]_{M})\in V_{R+\epsilon}$ for all $[v]_{M} \in V^{M}_R$ and $M\geq M_2$; and for all $[v]_{M}\in V^{M}_R$, $M\geq 1$ the following inequality is satisfied:
\[ \limsup_{M\to \infty} \Big ( \mathcal{J}(\mathcal{P}_M([v]_{M})) -\mathcal{I}_M([v]_{M}) \Big ) \leq 0.
\]
\item[iii] the following inequalities are satisfied:
  \[
    \limsup_{\epsilon \to 0} \mathcal{J}_*(\epsilon) \geq \mathcal{J}_*,
    \quad \liminf_{\epsilon \to 0} \mathcal{J}_*(-\epsilon) \leq \mathcal{J}_*,
  \]
where $\mathcal{J}_*(\pm\epsilon)=\inf_{V_{R\pm \epsilon}}\mathcal{J}(u)$.
\end{description}
\end{lemma}
\subsection{Properties of Interpolation and Discretization Maps}
The following properties are necessary to validate the hypotheses of Lemma~\ref{lem:Vasil'ev}
\begin{lemma}\label{mappings}\cite{abdulla13,abdulla15}
 For arbitrary sufficiently small $\epsilon > 0$ there exists $n=n(\epsilon)$ such that
\begin{gather}
\mathcal{Q}_n(v) \in \discreteControlSet{}, \quad \text{for all}~v\in
V_{R-\epsilon} \quad \text{and}~ n>n(\epsilon),\label{eq:mappings-1}
\\
\mathcal{P}_n([v]_n) \in V_{R+\epsilon}, \quad \text{for all}~[v]_n\in \discreteControlSet{} \quad \text{and}~ n>n(\epsilon)\label{eq:mappings-2}.
\end{gather}
\end{lemma}
\begin{corollary}\label{lipschitz}\cite{abdulla13}
  Let either $[v]_n \in \discreteControlSet{}$ or $[v]_n = {\Q}_n(v)$ for $v \in V_R$.
Then for sufficiently large \(n\),
\begin{equation}\label{eq:lipschitz-1}
  |s_k-s_{k-1}|
  \leq C' \tau,
  \quad k=1,2,\ldots,n
\end{equation}
where $C'$ is independent of $n$.
\end{corollary}
Note that for the step size $h_i$ we have one of the three possibilities: $h_i=h$, or $h_i=\overline{h}$, or $h_i \leq |s_{k}-s_{k-1}|$ for some $k$.
Hence, from~\eqref{eq:htau} and~\eqref{eq:lipschitz-1}, it follows that
\begin{equation}\label{eq:lipschitz-2}
\Delta = O(\sqrt{\tau}),
\quad \text{as}~ \tau \rightarrow 0.
\end{equation}
\subsection{Energy Estimates for the Discrete State Vector}
We first recall energy estimates proven for the discrete state vector due to~\cite{abdulla15}.
\begin{lemma}[\cite{abdulla15}, Thm.\ 3.1, p.\ 13]\label{thm:FEE}
For all sufficiently small $\tau$, the discrete state vector $[u([v]_n)]_n$ satisfies the following stability estimation:
\begin{gather}
  \max_{0\leq k \leq n} \sum_{i=0}^{N-1} h_i
  u_i^2(k)+\sum_{k=1}^{n}\tau\sum_{i=0}^{N-1}h_i u_{ix}^2(k)  \nonumber
  \\
  \leq C \Big (
  \norm{ \phi^n }_{L_2(0,s_0)}^2
  + \norm{ g^n }_{L_2(0,T)}^2
  + \norm{ f^n}_{L_2(D)}^2
  + \norm{ \gamma(s^n(t),t)(s^n)'(t) }_{L_2(0,T)}^2
  \nonumber
  \\
  + \norm{ \chi(s^n(t),t) }_{L_2(0,T)}^2
  + \sum_{k=1}^{n-1}\Ind_{+}(s_{k+1}-s_{k}) \sum_{i=m_j}^{m_{j_{k+1}}-1} h_i
  u_i^2(k)
  \Big), \label{eq:FEE}
\end{gather}
where $C$ is independent of $\tau$ and $\Ind_{+}$ is an indicator function of the positive semiaxis.
\end{lemma}
\begin{lemma}[\cite{abdulla15}, Thm.\ 3.2, p.\ 15]\label{thm:FCT}
Let $[v]_n \in \discreteControlSet{}$, $n=1,2,\ldots$ be a sequence of discrete controls and the sequence $\{\mathcal{P}_n([v]_n)\}$ converges weakly in $\controlSpaceWeaklyConverge{}$ to $v=(\controlVars{})$ (and hence with ($\controlVarsStronglyConvergeWithN{}$), converging strongly in $\controlSpaceStronglyConverge{}$).
Then the sequence $\bk{u^\tau}$ converges as $\tau \to 0$ weakly in
$B_2^{1,0}(\Omega)$ to weak solution $u \in V_2^{1,0}(\Omega)$ of the
problem~\eqref{eq:intro-pde}--\eqref{eq:intro-stefan-cond}, i.e.\ to the solution of the integral identity.
Moreover, $u$ satisfies the energy estimate
\begin{gather*}
  \norm{ u }_{V_2^{1,0}(D)}^2 \leq C \Big(
  \norm{ \phi }_{L_2(0,s_0)}^2
  + \norm{ g }_{L_2(0,T)}^2
  + \sup_n \norm{ f^n }_{L_2(D)}^2
  + \norm{ \gamma }_{B_2^{1,0}(D)}^2
  \\
  + \norm{ \chi }_{B_2^{1,0}(D)}^2
  \Big).
\end{gather*}

\end{lemma}
Let given discrete control vector $[v]_n$, along with discrete state vector $[u([v]_n)]_n$, the vector
\[
  [\tilde{u}([v]_n)]_n = (\tilde{u}(0), \tilde{u}(1), \ldots, \tilde{u}(n))
\]
is defined as
\[
  \tilde{u}_i(k) =
  \begin{cases}
  u_i(k) \quad 0 \leq i \leq m_j, \\
  u_{m_j}(k) \quad m_k < u \leq N, k=\overline{0,N}
  \end{cases}
\]
\begin{lemma}[\cite{abdulla15}, Thm.\ 3.3, p.\ 22]\label{thm:SEE}
For all sufficiently small $\tau$ discrete state vector $[u([v]_n)]_n$ satisfies the following stability estimation:
\begin{gather}
  \max_{1 \leq k \leq n}\sum_{i=0}^{m_j-1} h_i \tilde{u}_{ix}^2(k)
  +\tau \sum_{k=1}^{n}\sum_{i=0}^{m_j-1}h_i \tilde{u}_{i\bar{t}}^2(k)
  +\tau^2 \sum_{k=1}^{n}\sum_{i=0}^{m_j-1}h_i \tilde{u}_{ix\bar{t}}^2(k)
\leq C \bigg[
\norm{ \phi^n}_{L_2(0,s_0)}^{2} \nonumber
\\
+ \norm{ \phi}_{B_2^1(0,s_0)}^{2}
+\norm{g^n}_{B_2^{\frac{1}{4}}(0,T)}^{2}
+\norm{ \gamma(s^n(t),t)(s^n)'(t)}_{B_2^{\frac{1}{4}}(0,T)}^{2}
+\norm{ \chi(s^n(t),t)}_{B_2^{\frac{1}{4}}(0,T)}^{2} \nonumber
\\
+\norm{ f^n}_{L_{2}(D)}^{2}\bigg], \label{eq:SEE}
\end{gather}
where $C$ is independent of $\tau$.
\end{lemma}

\begin{lemma}[\cite{abdulla15}, Thm.\ 3.4, p.\ 25]\label{thm:SCT}
Let $[v]_n \in \discreteControlSet{}$, $n=1,2,\ldots$ be a sequence of discrete
controls and the sequence $\{\mathcal{P}_n([v]_n)\}$ converges weakly in
$\controlSpaceWeaklyConverge{}$ to $v=(\controlVars{})$.
For any $\delta > 0$ define
\[
  \Omega'=\Omega \cap \bk{x<s(t) - \delta,~ 0 < t < T}.
\]
Then the sequence $\bk{\hat{u}^\tau}$ converges as $\tau \to 0$ weakly in
$B_2^{1,1}(\Omega)$ to weak solution $u \in B_2^{1,1}(\Omega')$ of the problem~\eqref{eq:intro-pde}--\eqref{eq:intro-stefan-cond}, i.e.\ to the solution of the integral identity.
Moreover, $u$ satisfies the energy estimate
\begin{gather}
  \norm{ u }_{B_2^{1,1}(\Omega)}^2 \leq C \Big(
  \norm{ \phi }_{B_2^1(0,s_0)}^2
  + \norm{ g }_{B_2^{\frac{1}{4}}(0,T)}^2
+ \sup_n \norm{ f^n }_{L_2(D)}^2
+ \norm{ \gamma }_{B_2^{1,1}(D)}^2 \nonumber
\\
 + \norm{ \chi }_{B_2^{1,1}(D)}^2 \Big).\label{eq:second-energy-est-conseq}
\end{gather}
\end{lemma}
\endgroup

\begingroup
\def\sbar{\overline{s}}
\def\ubar{\overline{u}}
\def\fbar{\overline{f}}
\def\ybar{\bar{y}}
\def\stilde{\widetilde{s}}
\def\utilde{\widetilde{u}}
\def\ftilde{\widetilde{f}}
\def\shat{\widehat{s}}
\def\ustar{u_{*}}
\def\Omegahat{\widehat{\Omega}}
\newcommand\tildebar[1]{\widetilde{\overline{#1}}}
\def\utildebar{\tildebar{u}}
\def\solnspace{B_{2, x, t}^{5/2 + 2\alpha, 5/4 + \alpha}}
\def\densityspace{B_{2, x, t}^{1/2 + 2\alpha, 1/4 + \alpha}}
\def\ivspace{B_2^{3/2 + 2\alpha}}
\def\uxtrspace{B_2^{1 + \alpha}}
\def\uxxtrspace{B_2^{1/2 + \alpha}}
\def\uxxxtrspace{B_2^{\alpha}}
\def\controlVars{s, g, f}
\def\controlVarsWithDelta{{\Delta{} s}, {\Delta{} g}, {\Delta{} f}}
\def\controlVarsWithBar{\bar{s}, \bar{g}, \bar{f}}

\def\scontrolspace{B_2^2}
\def\gcontrolspace{B_2^{1/2 + \alpha}}
\def\fcontrolspace{B_{2, x, t}^{1, 1/4 + \alpha}}

\def\controlSpace{\tilde{H}}
\def\controlSpaceFull{\scontrolspace(0,T)\times{} \gcontrolspace(0,T)\times{}
  \fcontrolspace(D)}

\def\muspace{B_2^{1/4}}
\def\adjointsolnspace{B_{2, x, t}^{2, 1}}
\def\chigammaspace{B_{2, x, t}^{3/2 + 2\alpha^*,3/4 + \alpha^*}}
\subsection{$B_{2,x,t}^{2\l, \l}(Q_T)$-Solutions and Trace Results}\label{sec:besov-space-solutions}
In this section we recall the results of Solonnikov~\cite{solonnikov64} on
existence and energy estimates for solutions of linear parabolic equations in
parabolic Besov spaces \(B_{2,x,t}^{2\l, \l}\), as well as trace results for
Besov functions.
Consider the problem
\begin{equation}
  \begin{gathered}
    a u_{xx} + b u_{x} + c u - u_{t}=f~\text{in}~Q_{T},
    \\
    a(0,t)u_{x}(0,t)=\chi_{1}(t),~0 \leq t \leq T,
    \\
    a(1,t) u_{x}(1,t) = \chi_{2}(t),~0 \leq t \leq T,
    \\
    u(x,0)=\phi(x),~0 \leq x \leq 1.
  \end{gathered}\label{eq:solonnikov-problem}
\end{equation}
Let $\l>1$ be fixed.
\begin{lemma}\label{thm:solonnikov-solution-existence}~\cite[\S7,~Thm.\ 17]{solonnikov64}
  Suppose that
  \[
    \begin{gathered}
      a,b,c \in C_{x,t}^{2\l^*-2,\l^*-1}(Q_T),~\text{arbitrary}~\l^*>\l
      \\
      f \in B_{2,x,t}^{2\l-2,\l-1}(Q_T),\quad
      \phi \in B_{2}^{2\l-1}(0,1),\quad
      \chi_1,\chi_2 \in B_{2}^{\l-\frac{3}{4}}(0,T)
    \end{gathered}
  \]
    and the consistency condition of order
    $k = \left[ \l-\frac{5}{4}\right]$
    holds; that is,
    \[
        \D{^j (au_{x})}{x^j}(0,0) = \frac{d^j \chi_{1}}{d t^j}(0),
        \qquad \D{^j (au_{x})}{x^j}(1,0) = \frac{d^j \chi_{2}}{d t^j}(0),\quad j=0,\ldots,k
    \]
    Then the solution $u$ of~\eqref{eq:solonnikov-problem} satisfies the energy estimate
    \begin{gather}
        \norm{u}_{B_{2,x,t}^{2\l,\l}(Q_T)} \leq C \Big[
            \norm{f}_{B_{2,x,t}^{2\l-1,\l-1}(Q_T)}
            + \norm{\phi}_{B_2^{2\l-1}(0,1)}
            + \norm{\chi_1}_{B_{2}^{\l-\frac{3}{4}}(0,T)} \nonumber
            \\
            + \norm{\chi_2}_{B_{2}^{\l-\frac{3}{4}}(0,T)}
        \Big]\label{eq:solonnikov-energy-est}
    \end{gather}
\end{lemma}
In particular, energy estimate~\eqref{eq:solonnikov-energy-est} implies the existence and uniqueness of the solution in $B_{2,x,t}^{2\l,\l}(Q_T)$.

\begin{lemma}\label{thm:solonnikov-traces}~\cite[\S 4, Thm.\ 9]{solonnikov64}
    For a function $u \in B_{2,x,t}^{2\l,\l}(Q_T)$, the following bounded embeddings of traces hold: for any fixed $0 \leq t \leq T$,
    \begin{gather*}
        u(\cdot,t) \in B_2^{2\l-1}(0,1)~\text{when}~\l > 1/2.
        \intertext{For any fixed $0\leq x \leq 1$,}
        u(x,\cdot) \in B_2^{\l-1/4}(0,T)~\text{when}~\l > 1/4,
        \\
        u_x(x,\cdot) \in B_2^{\l-3/4}(0,T)~\text{when}~\l > 3/4,
        \\
        u_{xx}(x,\cdot) \in B_2^{\l-5/4}(0,T)~\text{when}~\l > 5/4.
    \end{gather*}
  \end{lemma}

  \subsection{Consequences of Energy Estimates and
    Embeddings}\label{sec:conseq-of-energy-estimates-and-embeddings}
For a given control vector $v=(\controlVars{}) \in W_R$ transform the domain $\Omega$ to the cylindrical domain $Q_T$
by the change of variables $y = x / s(t)$.
Let $d = d(x, t)$, $(x, t) \in \Omega$ stand for any of $u,a,b,c,f,\gamma,\chi$, define the function $\tilde{d}$ by
\begin{gather*}
    \tilde{d}(x,t) = d\big(x s(t), t\big),~\text{and}~
    \tilde{\phi}(x) = \phi\big( x s(0)\big).
\end{gather*}
The transformed function $\utilde$ is a \emph{pointwise a.e.} solution of the Neumann problem
\begin{gather}
    \tilde{L}\tilde{u}:=\frac{1}{s^2}\big(\tilde{a} \utilde_y\big)_y + \frac{1}{s}\big(\tilde{b} + y s'\big) \utilde_y + \tilde{c} \utilde - \utilde_{t} = \tilde{f}, ~\text{in}~Q_T\label{eq:pde-flat}
    \\
    \utilde(x,0) = \tilde{\phi}(x), ~0 \leq x \leq 1,\label{eq:pde-init-flat}
    \\
    \tilde{a}(0, t) \utilde_y(0, t) = g(t)s(t), ~0 \leq t \leq T,\label{eq:pde-bound-flat}
    \\
    \tilde{a}(1, t) \utilde_y(1, t) = \tilde{\chi}(1, t) s(t) - \tilde{\gamma}(1,t)s'(t)s(t),
    ~0 \leq t \leq T.\label{eq:pde-free-flat}
\end{gather}
After~\cite[Lem.\ 6.4]{abdulla16}, we have
\begin{lemma}\label{lem:j-well-defined}
  For fixed $v \in W_R$, there exists a unique solution $u \in B_2^{2,1}(\Omega)$ of the Neumann problem~\eqref{eq:intro-pde}--\eqref{eq:intro-stefan-cond} for which the transformed function $\utilde$ solves~\eqref{eq:pde-flat}--\eqref{eq:pde-free-flat} and satisfies the following energy estimate
  \begin{gather}
    \norm{\utilde}_{\solnspace(Q_T)} \leq C \Big(
    \norm{f}_{\fcontrolspace(D)}
    + \norm{\phi}_{\ivspace(0,s_0)} \nonumber
    \\
    +\norm{g}_{\uxxtrspace(0,T)}
    +\norm{\chi}_{\chigammaspace(D)}
    +\norm{\gamma}_{\chigammaspace(D)}
    \Big)\label{eq:energy-est-utilde}
  \end{gather}
  where $\alpha^*>\alpha$ is arbitrary.
\end{lemma}
\begin{lemma}\label{lem:adjoint-solution-exists}
  For fixed $v \in W_R$, given the corresponding state vector $u=u(x,t;v)$, there exists a unique solution $\psi \in \adjointsolnspace(\Omega)$ of the adjoint problem~\eqref{eq:adjoint-problem} and the following energy estimate is valid
  \begin{equation}
    \begin{gathered}
      \norm{\psi}_{\adjointsolnspace(\Omega)} \leq C \Big(
      \norm{f}_{L_2(\Omega)}
      + \norm{\phi}_{B_2^1(0,s_0)}
      + \norm{g}_{\uxxtrspace(0,T)}
      + \norm{\chi}_{B_{2,x,t}^{1+2\alpha,\frac{1}{2}+\alpha}(D)}
      \\
      + \norm{s}_{\scontrolspace(0,T)} \norm{\gamma}_{B_{2,x,t}^{1+2\alpha,\frac{1}{2}+\alpha}(D)}
      + \norm{w}_{B_2^1(0,s(T))}
      + \norm{\mu}_{\muspace(0,T)}
      + \norm{u - u_*}_{L_2(\Omega)}
      \Big).
    \end{gathered}\label{eq:psi-energy-est-1}
  \end{equation}
\end{lemma}
Take $\utilde$, $\tilde{a}$, etc.\ as in the proof of Lemma~\ref{lem:j-well-defined}, transform the domain $\Omega_{\bar{s}}$ to $Q_T$ by taking
$\ybar = x / \sbar(t)$, etc.\ in a similar way, and define $\tildebar{u}$,
$\tildebar{a}$, etc.\@
For $d$ standing for any of $u$, $f$, $a$, $b$, $c$, $\gamma$, $\chi$, denote by
\[
  \Delta \tilde{d}(y,t) := \tildebar{d}(y,t)-\tilde{d}(y,t),~(y,t) \in Q_T.
\]
\begin{lemma}\label{lem:deltau-main-est}
Under conditions~\eqref{eq:frechet-result-datacond},
\begin{equation}
  \norm{\Delta \utilde}_{\solnspace(Q_T)} \to 0
  ~\text{as}~ {\Delta v} \to 0~\mathrm{in}~\controlSpace{}.\label{eq:deltau-vanishes}
\end{equation}
Moreover,
\begin{gather}
  \norm{\Delta \utilde}_{V_2^{1,0}(Q_T)}
  \leq C \big(
  \norm{\Delta s}_{\scontrolspace(0,T)}^{1/2+\alpha^*}
  + \norm{\Delta g}_{L_2(0,T)}
  + \norm{\Delta f}_{L_2(D)}
  \big).\label{eq:v210-deltau-energy-est}
\end{gather}
\end{lemma}
Lemmas~\ref{lem:j-well-defined},~\ref{lem:adjoint-solution-exists}, and~\ref{lem:deltau-main-est} follow
easily from the proofs of the corresponding results in~\cite{abdulla17,abdulla16}.
\endgroup
\section{Proofs of Main Results}\label{sec:proof-main-results}
\subsection{Proof of Theorem~\ref{thm:existence-convergence}}

\begingroup
\def\scontrolspace{B_2^2}
\def\sdiscretecontrolspace{b_2^2}
\def\gcontrolspace{B_2^1}
\def\gdiscretecontrolspace{b_2^1}
\def\fcontrolspace{L_2}
\def\fdiscretecontrolspace{{\l}_{2}}

\def\controlVars{s, g, f}
\def\controlVarsWithN{\controlVarsSup{n}}
\def\controlVarsWithSupStar{\controlVarsSup{*}}
\def\controlVarsWithDelta{{\Delta{} s}, {\Delta{} g}, {\Delta{} f}}
\def\controlVarsWithdelta{{\delta{} s}, {\delta{} g}, {\delta{} f}}
\def\controlVarsWithTilde{\tilde{s}, \tilde{g}, \tilde{f}}
\def\controlVarsWithBar{\bar{s}, \bar{g}, \bar{f}}
\def\controlVarsWithNSub{\controlVarsSub{n}}
\def\controlVarsWithStar{\controlVarsSub{*}}
\def\controlVarsWithEpsilon{\controlVarsSub{\epsilon}}

\def\controlSpace{H}
\def\controlSpaceFull{\scontrolspace(0,T)\times{} \gcontrolspace(0,T)\times{}
  \fcontrolspace(D)}

\def\controlSpaceWeaklyConverge{\scontrolspace(0,T) \times{} \gcontrolspace(0,T) \times{} \fcontrolspace(D)}

\def\controlVarsStronglyConverge{s, g}
\def\controlVarsStronglyConvergeWithTilde{\tilde{s}, \tilde{g}}
\def\controlVarsStronglyConvergeWithSupStar{s^*, g^*}
\def\controlVarsStronglyConvergeWithN{s^n, g^n}
\def\controlSpaceStronglyConverge{B_2^1(0,T) \times{} L_2(0,T)}

\def\discreteControlVars{[s]_n, [g]_n, [f]_{nN}}
\def\discreteControlVarsList{$[s]_n$, $[g]_n$, $[f]_{nN}$}
\def\discreteControlSet{V_R^n}

\def\P{\mathcal{P}}
\def\Q{\mathcal{Q}}
\def\probIn{\mathcal{I}_n}
\def\fIn{I_n}
\begin{proof}[Proof of Part 1.]
  Let $\bk{v_{n}}\in {V_R}$ be a minimizing sequence for $\J$; that is,
  \[
    \lim_{n\to\infty} \J(v_{n}) = J_{*}
  \]
  Note that since $J(v) \geq 0$, $J_{*}\geq 0$.
  Since ${V_R}$ is bounded in the Hilbert space $\controlSpace{}$, $v_{n}=(\controlVarsWithNSub)$ is weakly precompact in $\controlSpaceWeaklyConverge{}$.
  Assume that $v_{n}\to v=(\controlVars{})\in {V_R}$ weakly in $\controlSpaceWeaklyConverge{}$,
  and hence $(\controlVarsStronglyConverge{})$ converge strongly in
  $\controlSpaceStronglyConverge{}$.
  Let $u_{n},u \in B_{2}^{1,1}(D)$ be the corresponding solutions to the Neumann problem~\eqref{eq:intro-pde}--\eqref{eq:intro-stefan-cond} in $B_{2}^{1,1}(\Omega_{n})$ and $B_{2}^{1,1}(\Omega)$, respectively, where
  \[
    \Omega_{n}=\bk{(x,t):0<x<s_{n}(t),~0<t<T}
  \]
  and $u_n$ and $u$ are extended to $B_2^{1,1}(D)$ such that
  \[
    \norm{u_n}_{B_2^{1,1}(D)} \leq C \norm{u_n}_{B_2^{1,1}(\Omega_n)},\quad
    \norm{u}_{B_2^{1,1}(D)} \leq C \norm{u}_{B_2^{1,1}(\Omega)}
    \]
  $u_{n}$ and $u$ satisfy the
  estimate~\eqref{eq:second-energy-est-conseq} with $(g_{n}, f_n)$ and
  $(g, f)$ respectively.
  Since $v_{n}\in V_{R}$, $u_{n}$ is in fact uniformly bounded in
  $B_{2}^{1,1}(D)$.
  Considering the sequence
  \[
    \Delta u = \Delta u_{n}= u_{n}-u,
  \]
  from Lemma~\ref{thm:SCT} we have the estimate
  $\norm{\Delta u}_{B_{2}^{1,1}(D)} \leq C$ uniformly with respect to $n$.
  Therefore, $\bk{\Delta u}$ is weakly precompact in $B_{2}^{1,1}(D)$.
  Without loss of generality, assume that $u_{n}-u$ converges weakly in $B_{2}^{1,1}(D)$ to an element $w \in B_{2}^{1,1}(D)$.
  Assume temporarily that the fixed test function $\Phi \in C^1(\bar{D})$.
  Considering the integral identities satisfied by $u_{n}$ and $u$,
  \begin{gather}
    0=\int_{0}^{T}\int_{0}^{s_{n}(t)}\big[
    au_{n,x}\Phi_{x}-b_n u_{n,x}\Phi - c_n u_{n}\Phi + u_{n,t}\Phi + f_n \Phi
    \big] \,dx\,dt\nonumber
    \\
    +\int_{0}^{T}\left[
      \gamma\big(s_{n}(t),t\big)s_{n}'(t)-\chi\big( s_{n}(t),t\big)
    \right]\Phi(s_{n}(t),t)\,dt + \int_{0}^{T}g_{n}(t)\Phi(0,t)
    \,dt\nonumber
    \intertext{and}
    0=\int_{0}^{T}\int_{0}^{s(t)}\big[
      au_{x}\Phi_{x}-bu_{x}\Phi-cu\Phi +
      u_{t}\Phi + f\Phi
      \big]\, dx\,dt\nonumber
    \\
    +\int_{0}^{T}\left[
      \gamma\big(s(t),t\big)s'(t)-\chi\big( s(t),t\big)
      \right]\Phi(s(t),t)\,dt + \int_{0}^{T}g(t)\Phi(0,t)\,dt\nonumber
  \end{gather}
  respectively, subtracting one from the other we see that $\Delta u = u_{n}-u$ satisfies
  \begin{gather}
    0=\int_{0}^{T}\int_{0}^{s(t)}\big[
    a{\Delta u}_{x}\Phi_{x}-b {\Delta u}_{x}\Phi - c {\Delta u}\Phi + {\Delta u}_{t}\Phi
    \big]\, dx\,dt\nonumber
    \\
    + I_1 + I_2 + I_3 + I_4 + I_5,\label{eq:identity-for-limitpt-w:disc}
    \intertext{where}
    I_1 := \int_{0}^{T}\int_{0}^{s(t)}\left(f_n - f\right) \Phi \,dx\,dt,\nonumber
    \\
    I_2 := -\int_{0}^{T}\int_{s_{n}(t)}^{s(t)}\big[ au_{n,x}\Phi_{x}-b_n u_{n,x}\Phi -
      c_n u_{n}\Phi + u_{n,t}\Phi + f_n \Phi\big]\, dx\,dt,\nonumber
    \\
    I_3 := \int_{0}^{T}\left[ \gamma\big(s_{n}(t),t\big)s_{n}'(t)-\chi\big(
    s_{n}(t),t\big)\right]\left(  \Phi(s_{n}(t),t)-\Phi(s(t),t)\right)\,dt,\nonumber
    \\
    I_4 := \int_{0}^{T}\Big{\{}
    \left[ \gamma s_{n}'(t)-\chi \right]_{x=s_n(t)}
    -\left[ \gamma s'(t)-\chi\right]_{x=s(t)}
    \Big{\}}\Phi\Big(s(t),t)\,dt,\nonumber
    \\
    I_5:= \int_{0}^{T}\left[ g_n(t)-g(t) \right]\Phi(0,t)\,dt,\nonumber
  \end{gather}

  for arbitrary fixed $\Phi \in C^1(\bar{D})$.
  Each of the terms $I_{1},\ldots,I_5$ vanish as $n\to\infty$.
  By weak convergence of $f_n$ to $f$ in $L_2(\Omega)$, it follows that
  $\lnorm{I_1} \to 0$ as $n \to \infty$.
  Each term in $I_2$ is handled using CBS inequality as well:
  \begin{gather}
    \lnorm{\int_{0}^{T}\int_{s(t)}^{s_{n}(t)} au_{n,x}\Phi_{x} \,dx \,dt}
    \leq M \norm{\Phi_x}_{C(D)} \norm{s_n-s}_{C[0,T]}^{1/2} \norm{u_n}_{B_2^{1,0}(D)}\to 0\nonumber
  \end{gather}
  as \(n \to \infty{}\) due to uniform boundedness of $u_{n}\in B_{2}^{1,0}(D)$ and uniform convergence of $s_{n}\to s$ on $[0,T]$.
  Treating each term in $I_2$ similarly, it follows that $\lnorm{I_{2}}\to 0$ as
  $n\to\infty$.
  Similarly, CBS inequality, continuity of the $L_2$ norm with respect to shift and uniform convergence of $s_n \to s$ imply $ \lnorm{I_{3}} \to 0$ and $\lnorm{I_4} \to 0$ as $n\to\infty$.
  Lastly, convergence of $g_{n}\to g$ strongly in $L_{2}(0,T)$ implies $\lnorm{I_{5}}\to 0$ as $n \to \infty$.

  Therefore, passing to the limit as $n\to\infty$ in~\eqref{eq:identity-for-limitpt-w:disc} we see that the limit point $w$ satisfies
  \begin{equation}
    0=\int_{0}^{T}\int_{0}^{s(t)}\big[
      a  w_{x}\Phi_{x}-b w_{x}\Phi-c w \Phi + w_{t}\Phi
      \big]\, dx\,dt,\quad \forall \Phi \in C^1(\bar{D}).\nonumber
  \end{equation}
  By extension of arbitrary $\Phi \in B_{2}^{1,1}(\Omega)$ to $B_2^{1,1}(D)$ and the density of $C^1(\bar{D})$ in $B_2^{1,1}(D)$, it follows that $w$ solves the Neumann problem~\eqref{eq:intro-pde}--\eqref{eq:intro-stefan-cond} with $f=g=\gamma=\chi\equiv 0$.
  By the uniqueness of the solution to the Neumann problem
  it follows that $u_{n} \to u$ weakly in $B_{2}^{1,1}(D)$ (and hence \(u_n \to
  u\) strongly in \(L_2(D)\)).
  By the Sobolev trace theorem~\cite{besov79,besov79a,nikolskii75}, it follows that
  \begin{equation}
    \norm{u_{n}(x,T)-u(x,T)}_{L_{2}(0,s_{n}(T))}\to 0,
    \quad \norm{u_{n}(s(t),t)-u(s(t),t)}_{L_{2}(0,T)} \to 0\nonumber
  \end{equation}
  as \(n\to\infty{}\).
  By Newton-Leibniz, CBS and Morrey's inequalities we have
  \begin{equation}
    \norm{u_{n}(s_{n}(t),t) - u_{n}(s(t),t)}_{L_{2}(0,T)}
    \leq \norm{s_n-s}_{C[0,T]}\norm{u_{n,x}}_{L_2(D)}
    \to 0.\nonumber
  \end{equation}
  and so
  \begin{gather}
    \norm{u_{n}(s_{n}(t),t) - u(s(t),t)}_{L_{2}(0,T)}\leq \norm{u_{n}(s_{n}(t),t) - u_{n}(s(t),t)}_{L_{2}(0,T)} \nonumber
    \\
    + \norm{u_{n}(s(t),t) - u(s(t),t)}_{L_{2}(0,T)} \to 0~\text{as}~n\to\infty,\nonumber
  \end{gather}
  which implies $J(v_n) \to J(v)$.
Lastly, consider
\begin{gather*}
  {\Delta P}_n := A_n \lnorm{
    \int_0^T\int_0^{s_n(t)} \subPlus{u_n(x,t)-u_{*}}^2 \,dx\,dt
  - \int_0^T\int_0^{s(t)} \subPlus{u(x,t)-u_{*}}^2 \,dx\,dt }
\\
\leq A_n \norm{\subPlus{u_n(x,t)-u_{*}}^2-\subPlus{u(x,t)-u_{*}}^2 }_{L_1(\Omega)}
\\
+\lnorm{\int_0^T\int_{s_n(t)}^{s(t)} (\subPlus{u_n(x,t)-u_{*}})^2 \,dx\,dt}.
\end{gather*}
By CBS inequality, it follows that
\[
{\Delta P}_n
\leq A_n \left[ \bar{P}_1 + \bar{P}_2\right],
\]
where
\begin{gather*}
\bar{P}_1 = \norm{ \subPlus{u_n(x,t)-u_{*}}-\subPlus{u(x,t)-u_{*}}
}_{L_2(\Omega)}
\\
\qquad \cdot\norm{ \subPlus{u_n(x,t)-u_{*}}+\subPlus{u(x,t)-u_{*}} }_{L_2(\Omega)},
\\
\bar{P}_2=\lnorm{ \int_0^T\int_{s_n(t)}^{s(t)} (\subPlus{u_n(x,t)-u_{*}})^2\, dx\,dt }.
\end{gather*}
As shown in Lemma~\ref{lem:trunc-props},
$\subPlus{u_n(x,t)-u_{*}},\subPlus{u_n(x,t)-v_{*}} \in B_2^{1,1}(\Omega)$
and, from~\eqref{eq:trunc-contractive} and CBS inequality it follows that
\[
  \bar{P}_1 \leq \norm{u_n(x,t) - u(x,t)}_{L_2(\Omega)}\left( \norm{u(x,t)}_{L_2(\Omega)}+\norm{u(x,t)}_{L_2(\Omega)} + 2 u_* \lnorm{\Omega}\right).
\]
By strong convergence of \(u_n \to u\) in \(L_2(D)\) it follows that $\bar{P}_1 \to 0$ as $n \to \infty$.
Similarly,
\begin{align*}
  \bar{P}_2
  & \leq \lnorm{ \int_0^T\int_{s_n(t)}^{s(t)} \lnorm{u_n(x,t)-u_{*}}^2\,
    dx\,dt }
    \\
                & \leq 2  \int_0^T\int_{0}^{\l} \lnorm{u_n(x,t)-u_n(x,t;v_{*})}\lnorm{u(x,t) + u(x,t; v_*)} \,dx \,dt
    \\
       &\quad
         + 2 \lnorm{ \int_0^T\int_{s_n(t)}^{s(t)} \lnorm{u(x,t)}^2 \,dx \,dt}
         + 2 \norm{s - s_n}_{L_1(0,T)} \lnorm{u_*}^2
\end{align*}
hence by CBS inequality, absolute continuity of the integral, and uniform
convergence of \(s_n \to s\), it follows that \(\bar{P}_2 \to 0\)
as \(n \to \infty{}\).
Since $J(v_n) \to J(v)$ and $\Delta P_n \to 0$ as $n \to \infty$, it follows
that $\J(v_n) \to \J(v) = J_*$.
Theorem is proved.
\end{proof}
The proof of Theorem~\ref{thm:existence-convergence}, part 2 is established
through three Lemmas.
The first is established in an analogous way to~\cite[Lem.\ 3.2]{abdulla15}
\begin{lemma}\label{lem:J-epsilon}\cite{abdulla15}
  Let $\J_*(\pm \epsilon) = \inf\limits_{V_{R\pm \epsilon}}\J(v)$, $\epsilon > 0$.
  Then

  \begin{gather}
    \lim\limits_{\epsilon \to 0} \J_*(\epsilon)=\J_* = \lim\limits_{\epsilon \to 0} \J_*(-\epsilon).\label{eq:J-epsilon}
  \end{gather}
\end{lemma}

\begin{lemma}\label{lem:Qn-convergence}
  For arbitrary $v=(\controlVars{}) \in V_R$,

  \begin{gather}
    \lim\limits_{n \to \infty} \probIn(\Q_n(v)) = \J(v).\nonumber
  \end{gather}
\end{lemma}

\begin{proof}
   Fix $v \in V_R$ and let $[v]_n = (\discreteControlVars{}) = \Q{}_n(v)$.
  Let $u = u(x,t;v)$ and $\big[ u([v]_n)\big]_n$ be the corresponding continuous
  and discrete state vector, respectively, and denote by $v^{n} =
    (\controlVarsWithN{}) = \P{}_{n}({[v]}_{n})$.
  By Sobolev embedding theorem, $s^{n}(t) \to s(t)$ uniformly on $[0,T]$.
  Let $\epsilon_m \downarrow 0$ be an arbitrary sequence, and define
  \[
    \Omega_m = \bk{(x,t):0 < x < s(t) - \epsilon_m, 0 < t \leq T}
  \]
  and fix $m>0$.

  In Lemma~\ref{thm:SEE} it was shown that $\bk{\hat{u}^{\tau}}$
  converges to $u$ weakly in $B_2^{1,1}(\Omega_m)$ for any fixed $m$; by the
  embeddings of traces, it follows that $\bk{\hat{u}^{\tau}(s(t)-\epsilon_m,t)}$
  and $\bk{\hat{u}^\tau(x, T)}$ converge to the corresponding traces
  $u(s(t)-\epsilon_m, t)$ and $u(x,T)$ weakly in $L_2(0,T)$ and $L_2(0,
    s(t)-\epsilon_m)$, respectively.
  As in~\cite[Lem.\ 3.3]{abdulla15}, it follows that the corresponding traces of
  \(u^\tau{}\) satisfy the same property, and hence
  \begin{equation}
    \lim_{n \to \infty} I_n(\Q_n(v))
    = J(v).\label{eq:convergence-lemma-main-terms-abdulla15}
  \end{equation}
  It remains to show that
  \begin{equation}
    \lim_{n\to\infty} P_k^n(\Q_n(v))
    = P_k(v).\label{eq:convergence-lemma-new-terms}
  \end{equation}

  Estimate the difference \(\lnorm{P_k^n(\Q{}_n(v_n)) - P_k(v)}\) as
  \begin{gather}
  \lnorm{P_k^n(\Q{}_n(v_n)) - P_k(v)}
  \leq A_k \left(P_1 + P_2+P_3\right),\label{eq:Qn-convergence-est-1}
  \\
  P_1 = \lnorm{\int_0^T \int_{0}^{s(t)-\epsilon_m} (\tilde{u}^\tau(x,t)-u_*)^2 - (u(x,t)-u_*)^2 \,dx \,dt}\nonumber
  \\
  P_2 = \lnorm{\int_{0}^{T} \int_{s(t)-\epsilon_m}^{s(t)} (u(x,t)-u_*)^2 \,dx \,dt}\nonumber
  \\
  P_3=\lnorm{\int_{0}^{T} \int_{s(t)-\epsilon_m}^{s^\tau(t)} (\tilde{u}^\tau(x,t)-u_*)^2 \,dx \,dt}.\nonumber
  \end{gather}

  First, fix \(m\), and estimate $P_1$ using CBS inequality and~\eqref{eq:trunc-contractive}:
  \begin{gather}
  P_1 \leq \int_{0}^{T}\int_{0}^{s(t)-\epsilon_m} \lnorm{(\tilde{u}^\tau(x,t)-u_*)^2 - (u(x,t)-u_*)^2} \,dx \,dt
  \nonumber
  \\
  \leq \norm{\tilde{u}^\tau(x,t)-u(x,t)}_{L_2(\Omega_m)}
  \norm{\tilde{u}^\tau(x,t) + u(x,t) - 2u_*}_{L_2(\Omega_m)}
  \nonumber
  \\
  \leq C \norm{\tilde{u}^\tau(x,t)-u(x,t)}_{L_2(\Omega_m)}.\nonumber
  \end{gather}
  The term on the right is bounded due to first energy estimate.
  Thus we need to show convergence of the term on the left:
  \begin{gather}
    \norm{\tilde{u}^\tau(x,t)-u(x,t)}_{L_2(\Omega_m)}
    \leq \norm{\hat{u}^\tau(x,t)-u(x,t)}_{L_2(\Omega_m)}
    +\norm{\tilde{u}^\tau(x,t)-\hat{u}^\tau(x,t)}_{L_2(\Omega_m)}\nonumber
  \end{gather}
  In Lemma~\ref{thm:SCT}, it was shown that \(\hat{u}^\tau \to u\) weakly in
  \(B_2^{1,1}(\Omega_m)\), and hence for fixed $m$, it follows that there is some $N_1=N_1(m)$ such that for $n>N_1$
  \[
    \norm{\hat{u}^\tau(x,t)-u(x,t)}_{L_2(\Omega_m)} \leq \frac{1}{m}.
  \]

  Denote by
  \[
    \hat{\imath} = \max \bigg{\{}i\leq N : -\epsilon_m \leq x_i - \max_{t_{k-1} \leq t \leq t_k} s(t) \leq -\frac{\epsilon_m}{2} \bigg{\}}
  \]
   and $s_k^m=x_{\hat{\imath}}$.
   Then
  \begin{gather}
    \norm{\tilde{u}^\tau(x,t)-\hat{u}^\tau(x,t)}_{L_2(\Omega_m)}^2
    \leq \int_{0}^{T} \int_{0}^{s_k^m} (\tilde{u}^\tau(x,t)-\hat{u}^\tau(x,t))^2\, dx\, dt
    \nonumber
    \\
    \leq 4 \tau^2 \sum_{k=1}^{n} \tau \sum_{i=0}^{\hat{\imath}-1} h_i u_{i,\bar{t}}^2(k) + \frac{4}{3} C\tau \sum_{k=1}^{n} \tau \sum_{i=0}^{\hat{\imath}-1} h_i u_{ix}^2(k-1)
    \nonumber
    \\
    + \frac{4}{3} \tau^2 \sum_{k=1}^{n} \tau \sum_{i=0}^{\hat{\imath}-1} h_i u_{i,\bar{t}}^2(k) + \frac{4}{9} C \tau^2 \sum_{k=1}^{n} \tau^2 \sum_{i=0}^{\hat{\imath}-1} h_i u_{ix,\bar{t}}^2(k).\label{eq:utildetau-minus-uhattau-L2Estimate}
  \end{gather}
   As in the estimation~\cite[eq.\ 102]{abdulla16}, it follows
   from~\eqref{eq:SEE} and~\eqref{eq:utildetau-minus-uhattau-L2Estimate} that for sufficiently large \(n\),
   \begin{gather}
    \norm{\tilde{u}^\tau(x,t)-\hat{u}^\tau(x,t)}_{L_2(\Omega_m)}^2
    \leq C_1 \tau \bigg( \sum_{k=1}^{n} \tau \sum_{i=0}^{m_j-1} h_i \tilde{u}_{ix}^2(k-1)
  \nonumber
  \\
  + \sum_{k=1}^{n} \tau \sum_{i=0}^{m_j-1} h_i \tilde{u}_{i,\bar{t}}^2(k)
  + \sum_{k=1}^{n} \tau^2 \sum_{i=0}^{m_j-1} h_i \tilde{u}_{ix,\bar{t}}^2(k)
  \bigg)
  \leq \frac{C_2}{m}\nonumber
\end{gather}
   where $C_2$ is independent of $n$, and
  \begin{gather*}
    C_1= \frac{16}{3}T+\frac{4}{9}CT+\frac{4}{3}C.
  \end{gather*}
  We estimate $P_3$ as
  \begin{gather}
    P_3 \leq  P_{3,1} + P_{3,2},\quad
    P_{3,1}:=\lnorm{\int_{0}^{T} \int_{s(t)-\epsilon_m}^{s^\tau(t)} (u^\tau(x,t)-u_*)^2 \,dx \,dt}
    \nonumber
    \\
    P_{3,2} := \lnorm{\int_{0}^{T} \int_{s(t)-\epsilon_m}^{s^\tau(t)}
      (\tilde{u}^\tau(x,t)-u_*)^2 - (u^\tau(x,t)-u_*)^2 \,dx \,dt}.\nonumber
  \end{gather}
  Estimate \(P_{3,1}\) as
  \begin{gather}
    P_{3,1}
    \leq \norm{u^\tau(x,t)-u_*}_{L_4(D)}^2
    T^{\frac{1}{2}}\norm{s^\tau-s+\epsilon_m}_{C[0,T]}^{\frac{1}{2}}\nonumber
    \\
    \leq \left( \norm{u^\tau}_{V_2^{1,0}(D)} + \lnorm{u_*}^2(\l T)^{\frac{1}{2}}\right) T^{\frac{1}{2}}\norm{s^\tau-s+\epsilon_m}_{C[0,T]}^{\frac{1}{2}}
    \leq C_3 \norm{s^\tau-s+\epsilon_m}_{C[0,T]}^{\frac{1}{2}}\nonumber
  \end{gather}
  due to embedding of $V_2^{1,0}(D)$ into \(L_4(D)\),
  estimate~\cite[eq.\ 106]{abdulla16}.
  By uniform convergence of \(s^\tau{}\) to \(s\), it follows that
  \[
    \lim_{n \to \infty} P_{3,1} \leq C_3 \sqrt{\epsilon_m}.
  \]
  By~\eqref{thm:FEE}, CBS and Minkowski inequalities,
  \begin{gather}
  P_{3,2}
  \leq \lnorm{\int_{0}^{T} \int_{s(t)-\epsilon_m}^{s^\tau(t)} \lnorm{\tilde{u}^\tau(x,t)-u^{\tau}(x,t)}\lnorm{\tilde{u}^{\tau}(x,t) + u^{\tau}(x,t) - 2u_{*}} \,dx \,dt}\nonumber
  \\
  \leq
  C_4 \norm{\tilde{u}^\tau-u^{\tau}}_{L_2(D)}\nonumber
  \\
   \leq \bigg(\frac{C C^2_4}{3}\bigg)^{\frac{1}{2} } \tau^\frac{1}{2} \bigg( \sum_{k=1}^{n} \sum_{i=0}^{N-1} \tau h_i u_{ix}^2(k)\bigg)^{\frac{1}{2}} \to 0\nonumber
\end{gather}
as $n \to \infty$.
It follows that
\begin{gather}
  0 \leq \limsup_{n \to \infty}  \lnorm{A_k\sum_{k=1}^n\sum_{i=0}^{m_{j_k}-1} \tau h_i \subPlus{u_i(k)-u_*}^2 - A_k\int_{0}^{T} \int_{0}^{s(t)} \subPlus{u(x,t)-u_*}^2 \,dx \,dt}
  \nonumber
  \\
  \leq \frac{2C + C_2}{m} + C_3 \epsilon_m^{\frac{1}{2}} + P_2\nonumber
\end{gather}
for all \(m\).
Passing to the limit as $m \to \infty$
establishes~\eqref{eq:convergence-lemma-new-terms}, which, together
with~\eqref{eq:convergence-lemma-main-terms-abdulla15} proves the Lemma.
\end{proof}

\begin{lemma}\label{lem:Pn-convergence}
  For arbitrary $[v]_n=(\discreteControlVars{}) \in V_R^n$,
  \begin{equation*}
    \lim_{n \to \infty} \probIn([v]_n) - \J(\P_n([v]_n)) = 0.
  \end{equation*}
\end{lemma}

\begin{proof}
  Let $[v]_n \in \discreteControlSet{}$ and $v^n = (\controlVarsWithN{}) = \P_n([v]_n)$.
Then $\bk{\P_n([v]_n)}$ is weakly precompact in $\controlSpace{}$; assume that the whole sequence converges to $\tilde{v}=(\controlVarsWithTilde{})$.
  Then $\tilde{v} \in V_R$, and moreover, Rellich-Kondrachov compactness theorem implies that $(\controlVarsStronglyConvergeWithN{}) \to (\controlVarsStronglyConvergeWithTilde{})$ strongly in $\controlSpaceStronglyConverge{}$; in particular, $s^n \to \tilde{s}$ uniformly on $[0,T]$.
  Write the difference $\J\big(\P_n([v]_n)\big) - \fIn\big([v]_n \big)$ as
  \[
    \fIn\big([v]_n \big) - \J\big(\P_n([v]_n)\big)
    = \fIn\big([v]_n \big) - \J\big(v^n\big)
    = \fIn\big([v]_n \big) - \J(\tilde{v}) + \J(\tilde{v}) - \J\big(v^n\big).
  \]
  By weak continuity of $\J$, we have
  \(
    \lim_{n\to\infty} \left(\J(\tilde{v}) - \J\big(v^n\big)\right) = 0
  \).
  It remains to be shown that
  \[
    \lim_{n\to\infty} \left(\fIn\big([v]_n \big) - \J(\tilde{v})\right) = 0.
  \]
  Since $\tilde{v} \in V_{R+\epsilon}$ for some $\epsilon > 0$, and by strong convergence of $\P_n({[v]}_n) \to \tilde{v}$, a nearly identical argument to the proof of Lemma~\ref{lem:Qn-convergence} establishes this result.
\end{proof}
By Lemmas~\ref{lem:J-epsilon}--\ref{lem:Pn-convergence} and~\ref{lem:Vasil'ev},
Theorem~\ref{thm:existence-convergence} is proved.\qed%
\endgroup
\subsection{Proof of Theorem~\ref{thm:gradient}}\label{sec:frechet-differentiability}

\begingroup
\def\sbar{\overline{s}}
\def\ubar{\overline{u}}
\def\fbar{\overline{f}}
\def\ybar{\bar{y}}
\def\stilde{\widetilde{s}}
\def\utilde{\widetilde{u}}
\def\ftilde{\widetilde{f}}
\def\shat{\widehat{s}}
\def\ustar{u_{*}}
\def\Omegahat{\widehat{\Omega}}
\newcommand\tildebar[1]{\widetilde{\overline{#1}}}
\def\utildebar{\tildebar{u}}
\def\solnspace{B_{2, x, t}^{5/2 + 2\alpha, 5/4 + \alpha}}
\def\densityspace{B_{2, x, t}^{1/2 + 2\alpha, 1/4 + \alpha}}
\def\ivspace{B_2^{3/2 + 2\alpha}}
\def\uxtrspace{B_2^{1 + \alpha}}
\def\uxxtrspace{B_2^{1/2 + \alpha}}
\def\uxxxtrspace{B_2^{\alpha}}
\def\controlVars{s, g, f}
\def\controlVarsWithDelta{{\Delta{} s}, {\Delta{} g}, {\Delta{} f}}
\def\controlVarsWithBar{\bar{s}, \bar{g}, \bar{f}}

\def\scontrolspace{B_2^2}
\def\gcontrolspace{B_2^{1/2 + \alpha}}
\def\fcontrolspace{B_{2, x, t}^{1, 1/4 + \alpha}}

\def\controlSpace{\tilde{H}}
\def\controlSpaceFull{\scontrolspace(0,T)\times{} \gcontrolspace(0,T)\times{}
  \fcontrolspace(D)}

\def\muspace{B_2^{1/4}}
\def\adjointsolnspace{B_{2, x, t}^{2, 1}}
\def\chigammaspace{B_{2, x, t}^{3/2 + 2\alpha^*,3/4 + \alpha^*}}
Note that $W_R$ is a closed, bounded, and convex subset of $\controlSpace{}$, so the left hand side of
the optimality condition~\eqref{eq:optimality-condition} is uniquely defined for Frechet
gradient $\J'(v)$ defined in the sense of definition~\ref{defn:Frechet}.
Define ${\Delta v} = (\controlVarsWithDelta{})$, $\bar{v} = v+{\Delta v} =
(\controlVarsWithBar{})$ such that $\bar{v} \in W_R$.
Let $\ubar(x,t) = \ubar(x,t,\bar{v})$ and
\begin{gather*}
  \shat = \min(s, \sbar),\quad 0 \leq t \leq T,\quad
  \Omegahat = \bk{(x,t):0 < x < \shat(t),~ 0 < t \leq T},
  \\
  \Delta u(x,t) = \ubar(x,t) - u(x,t)~\text{in}~\Omegahat.
\end{gather*}
Partition the time domain as $[0,T] = T_{1}\cup T_{2}$ where
\[
  T_1 =\bk{t \in [0,T]: \Delta s(t) <0},\quad
  T_2 = [0,T] \setminus T_1 = \bk{t \in [0,T]: \Delta s(t) \geq 0}.
\]
Let $\psi$ be a solution of the adjoint problem~\eqref{eq:adjoint-problem}.
Transforming $\Delta J$ as in~\cite{abdulla17,abdulla16}, it follows that
\begin{align}
  {\Delta J}(v)
  & = \int_0^T \big[2\beta_1(u-\mu)u_x \big]_{x=s(t)}{\Delta s}(t) \, dt
    +
    \int_{\Omegahat} \left[-2A_k \left( u(x, t) - \ustar \right)\right]{\Delta u} \,dx \,dt \nonumber
    \\
      & \quad + [\beta_0(u(s(T),T)-w(s(T)))^2 + 2 \beta_2(s(T)-s_*)]{\Delta s}(T) \nonumber
    \\
      & \quad -\int_{\Omega} \psi {\Delta f} \,dx \,dt
    + \int_0^T\Big[ \psi \big(\chi_x {\Delta s} - \gamma_x s' {\Delta s} - \gamma {\Delta s}' - (a u_x)_x {\Delta s}\big)\Big]_{x=s(t)} \,dt \nonumber
    \\
      & \quad - \int_{0}^T \psi(0,t) {\Delta g}(t) \, dt
      + o({\Delta v}).\label{eq:gradient-J-increment}
\end{align}
Write the increment \({\Delta P}_k := P_k(\bar{v}) - P_k(v)\) as
\begin{gather}
  {\Delta P}_k 
      = A_k  \left[ P_1 + P_2 \right],\nonumber
    \end{gather}
    where
    \begin{gather}
      P_1 :=
      \int_{\Omegahat} \lnorm{\subPlus{u(x,t) - \ustar + {\Delta u}(x,t)}}^2  - \lnorm{\subPlus{u(x,t) - \ustar}}^2 \,dx \,dt\nonumber
    \\
    P_2 := \int_{T_2} \int_{\shat(t)}^{\sbar(t)}
    \lnorm{\subPlus{\ubar(x,t) - \ustar}}^2 \,dx \,dt
    - \int_{T_1} \int_{\shat(t)}^{s(t)}
    \lnorm{\subPlus{u(x,t) - \ustar}}^2 \,dx \,dt.\nonumber
  \end{gather}
  By mean value theorem,
  \begin{gather}
    P_1 
    = 2 \int_{\Omegahat} \subPlus{u(x,t) - \ustar}{\Delta u}(x,t) \,dx \,dt + R_{1},\label{eq:Pk-term1-final}
    \\
    R_{1} := 2 \int_{\Omegahat}
    \left[
    \subPlus{u(x,t) - \ustar + \theta {\Delta u}(x,t)}
    - \subPlus{u(x,t) - \ustar}\right]{\Delta u}(x,t)
    \,dx \,dt\nonumber
  \end{gather}
  where \(0<\theta<1\) is in general a function of \((x,t)\).
  Similarly,
  \begin{gather}
    P_2 = \int_{T_2}
    \lnorm{\subPlus{\ubar(s(t) + \theta {\Delta s}(t),t) - \ustar}}^2
    {\Delta s}(t)
    \,dt\nonumber
    \\
    + \int_{T_1}
    \lnorm{\subPlus{u(s(t) - \theta {\Delta s}(t),t) - \ustar}}^2 {\Delta s}(t)
    \,dt\nonumber
    \\
    = \int_{0}^T
      \lnorm{\subPlus{u(s(t),t) - \ustar}}^2
      {\Delta s}(t)
      \,dt
    + R_{2} + R_{3} + R_{4},\label{eq:Pk-term2-final}
    \\
    R_{2} := \int_{T_2}
    \left[
      \lnorm{\subPlus{u(s(t),t) + {\Delta u}(s(t), t) - \ustar}}^2
      - \lnorm{\subPlus{u(s(t),t) - \ustar}}^2
      \right]
    {\Delta s}(t)
    \,dt,\nonumber
    \\
    R_{3} := \int_{T_2}
     \left[
       \lnorm{\subPlus{\ubar(s(t) + \theta {\Delta s}(t),t) - \ustar}}^2 -
    \lnorm{\subPlus{\ubar(s(t),t) - \ustar}}^2 \right]
    {\Delta s}(t)
    \,dt,\nonumber
    \\
    R_{4} := \int_{T_1}
    \left[
      \lnorm{\subPlus{u(s(t) - \theta {\Delta s}(t),t) - \ustar}}^2 - \lnorm{\subPlus{u(s(t),t) - \ustar}}^2
      \right]{\Delta s}(t)
    \,dt.\nonumber
  \end{gather}
  Combining~\eqref{eq:Pk-term1-final} and~\eqref{eq:Pk-term2-final} it follows
  that
  \begin{gather}
    {\Delta P}_k = 2 A_k \int_{\Omegahat} \subPlus{u(x,t) - \ustar}{\Delta u}(x,t)
    \,dx \,dt
    + A_k \int_{0}^T
    \lnorm{\subPlus{u(s(t),t) - \ustar}}^2
    {\Delta s}(t)
    \,dt\nonumber
    \\
    + R_{1} + R_{2} + R_{3} + R_{4}.\label{eq:Pk-final}
  \end{gather}
  By~\eqref{eq:gradient-J-increment},~\eqref{eq:Pk-final}, it follows that
  \begin{gather}
    {\Delta J}(v)
    = \int_0^T \big[2\beta_1(u-\mu)u_x \big]_{x=s(t)}{\Delta s}(t) \, dt
    -\int_{\Omega} \psi {\Delta f} \,dx \,dt\nonumber
    \\
    + [\beta_0(u(s(T),T)-w(s(T)))^2 + 2 \beta_2(s(T)-s_*)]{\Delta s}(T)
    - \int_{0}^T \psi(0,t) {\Delta g}(t) \, dt\nonumber
    \\
    + \int_0^T\Big[
      \psi \big(\chi_x {\Delta s} - \gamma_x s' {\Delta s} - \gamma {\Delta s}' - (a u_x)_x {\Delta s}\big)
      \Big]_{x=s(t)} \,dt \nonumber
    \\
    + A_k \int_{0}^T
    \lnorm{\subPlus{u(s(t),t) - \ustar}}^2
    {\Delta s}(t)
    \,dt
    + o({\Delta v}) + \sum_{i=1}^{4}R_i.\label{eq:gradient-full-increment}
  \end{gather}
  Estimate \(R_{1}\) using CBS inequality
  \begin{gather}
    \lnorm{R_{1}}
    \leq 2 \norm{\subPlus{u(x,t) - \ustar + \theta {\Delta u}(x,t)}
      - \subPlus{u(x,t) - \ustar}}_{L_2(\Omegahat)}
    \norm{\Delta u}_{L_2(\Omegahat)}.\nonumber
    \intertext{By~\eqref{eq:trunc-contractive}, it follows that}
    \lnorm{R_{1}} \leq 2 \norm{\Delta u}_{L_2(\Omegahat)}^2.\nonumber
  \end{gather}
  Write
  \begin{gather}
    \norm{\Delta u}_{L_2(\Omegahat)}^2
    = \int_{T_1} \int_0^{\sbar(t)}\lnorm{\ubar(x,t) - u(x,t)}^2 \,dx \,dt
    + \int_{T_2} \int_0^{s(t)}\lnorm{\ubar(x,t) - u(x,t)}^2 \,dx \,dt\nonumber
    \\
    \leq
      2 \l \bigg[
    \int_{0}^T \int_0^{1}\lnorm{{\Delta \utilde}(y,t)}^2 \,dy \,dt
    + \int_{T_1} \int_0^{1}\lnorm{u(y s(t), t)- u(y \sbar(t),t)}^2 \,dy \,dt\nonumber
    \\
    + \int_{T_2} \int_0^{1}\lnorm{\ubar(y s(t),t) - \ubar(y \sbar(t), t)}^2 \,dy \,dt
    \bigg].\nonumber
    \intertext{By Newton-Leibniz formula,}
    \norm{\Delta u}_{L_2(\Omegahat)}^2
    \leq 2 \l \bigg[
    \norm{\Delta \utilde}_{L_2(Q_T)}^2
    + \int_{T_1} \int_0^{1}\lnorm{\int_{y \sbar(t)}^{y s(t)} u_x(z, t) \,dz}^2 \,dy \,dt\nonumber
    \\
    + \int_{T_2} \int_0^{1}\lnorm{\int_{y\sbar(t)}^{ys(t)}\ubar_x(z,t) \,dz}^2 \,dy \,dt
    \bigg].\nonumber
    \end{gather}
    Take the change of variable \(\theta = z/ s(t)\) to derive
    \[
    \int_0^{1}\lnorm{\int_{y \sbar(t)}^{y s(t)} u_x(z, t) \,dz}^2 \,dy
    =
    \int_0^{1}\lnorm{\int_{y}^{y \sbar(t)/s(t)} \tilde{u}_x(z, t) \,dz}^2
    \,dy.\nonumber
    \]
    By boundedness of \(\tilde{u}_x\) for \(\tilde{u} \in
    \solnspace(Q_T)\) and~\eqref{eq:energy-est-utilde}, it follows that
    \begin{gather}
    \int_0^{1}\lnorm{\int_{y \sbar(t)}^{y s(t)} u_x(z, t) \,dz}^2 \,dy
    \leq C\norm{\tilde{u}}_{\solnspace(Q_T)}\int_0^{1}y^2
    \lnorm{\frac{\sbar(t)}{s(t)} - 1}^2 \,dy\nonumber
    \\
    =
    C \norm{\tilde{u}}_{\solnspace(Q_T)}\frac{1}{3}
    \lnorm{\frac{\sbar(t)}{s(t)} - 1}^2
    \leq \frac{C}{\delta} \lnorm{{\Delta s}(t)}^2,\nonumber
    \intertext{and hence}
    \norm{\Delta u}_{L_2(\Omegahat)}^2
    \leq 2 \l \bigg[
    \norm{\Delta \utilde}_{L_2(Q_T)}^2
    + C \norm{\Delta s}_{B_2^1(0,T)}^2
    \bigg].\nonumber
  \end{gather}
  Hence \(R_{1} = o({\Delta v})\).
  In \(R_{2}\), use CBS and Morrey inequalities to write
  \begin{gather}
    \lnorm{R_{2}}
      \leq \norm{\Delta s}_{B_2^1(0,T)}\bar{R}_{2} \norm{\subPlus{u(s(t),t) + {\Delta u}(s(t), t) - \ustar}
      + \subPlus{u(s(t),t) - \ustar}}_{L_2(T_2)},\label{eq:diff-est-r2term-v1}
      \intertext{where}
      \bar{R}_{2} := \norm{
      \subPlus{u(s(t),t) + {\Delta u}(s(t), t) - \ustar}
      - \subPlus{u(s(t),t) - \ustar}
    }_{L_2(T_2)}.\nonumber
    \end{gather}
    By~\eqref{eq:trunc-contractive}, it follows that
    \[
    \bar{R}_{2} \leq \norm{
      {\Delta u}(s(t), t)
    }_{L_2(T_2)}.
    \]
    By definition,
    \begin{gather}
    \norm{
      {\Delta u}(s(t), t)
    }_{L_2(T_2)}
    =\left( \int_{T_2} \lnorm{\utildebar\left(\frac{s(t)}{\sbar(t)},t\right) -
        \utilde(1, t)}^2 \,dt \right)^{1/2}\nonumber
    \\
    \leq \left( \int_{T_2}
      \lnorm{\utildebar\left(\frac{s(t)}{\sbar(t)},t\right) - \utildebar(1, t)}^2
      \,dt\right)^{1/2}
    +\left( \int_{T_2}
      \lnorm{{\Delta \utilde}(1, t)}^2
      \,dt \right)^{1/2}.\nonumber
    \intertext{By CBS inequality, trace embedding, and
      estimate~\eqref{eq:v210-deltau-energy-est}, it follows that}
    \bar{R}_{2}
    \leq \left( \int_{T_2} \lnorm{1-\frac{s(t)}{\sbar(t)}}
      \int_{s(t)/\sbar(t)}^1 \lnorm{\utildebar_x(z, t)}^2 \,dz
    \right)^{1/2}\nonumber
    \\
    + C \left( \norm{\Delta s}_{\scontrolspace(0,T)}^{1/2+\alpha^*}
      + \norm{\Delta g}_{L_2(0,T)}
      + \norm{\Delta f}_{L_2(D)} \right).\nonumber
    \end{gather}
    Hence, using~\eqref{eq:energy-est-utilde} and Morrey's inequality, it
    follows that
    \begin{equation}
      \bar{R}_{2}
      \leq C \norm{\Delta s}_{B_2^1(0,T)}^{1/2}
      + C \left( \norm{\Delta s}_{\scontrolspace(0,T)}^{1/2+\alpha^*}
        + \norm{\Delta g}_{L_2(0,T)}
        + \norm{\Delta f}_{L_2(D)} \right)\label{eq:diff-est-r2barterm-v1}
  \end{equation}
  Using~\eqref{eq:diff-est-r2barterm-v1} in~\eqref{eq:diff-est-r2term-v1} it follows that \(R_{2} = o(\Delta v)\) as \({\Delta v} \to 0\).
  Estimate \(R_{3}\) using Morrey's inequality and CBS inequality to derive
  \begin{gather}
    \lnorm{R_{3}}
    \leq C \norm{\Delta s}_{B_2^1(0,T)}
    \int_{T_2}
     \left[
       \lnorm{\subPlus{\ubar(s(t) + \theta {\Delta s}(t),t) - \ustar}}^2 -
    \lnorm{\subPlus{\ubar(s(t),t) - \ustar}}^2 \right]
  \,dt\nonumber
  \\
  \leq C \norm{\Delta s}_{B_2^1(0,T)} \bar{R}_{3}
  \norm{\subPlus{\ubar(s(t) + \theta {\Delta s}(t),t) - \ustar} +
    \subPlus{\ubar(s(t),t) - \ustar}}_{L_2(T_2)},\label{eq:diff-est-r3term-v1}
  \intertext{where}
  \bar{R}_{3} := \norm{\subPlus{\ubar(s(t) + \theta {\Delta s}(t),t) - \ustar} -
    \subPlus{\ubar(s(t),t) - \ustar}}_{L_2(T_2)}.\nonumber
  \intertext{By~\eqref{eq:trunc-contractive}, it follows that}
  \bar{R}_{3} \leq \norm{\ubar(s(t) + \theta {\Delta s}(t),t) -
    \ubar(s(t),t)}_{L_2(T_2)}.\nonumber
  \end{gather}
  By Newton-Leibniz formula, CBS, and Morrey's inequalities, it follows that
  \begin{gather}
  \bar{R}_{3}
  \leq C \norm{\Delta s}_{B_2^1(0,T)}\norm{\ubar}_{B_2^{1,0}(D)}.\label{eq:diff-est-r3barterm-v1}
\end{gather}
Using energy estimate~\ref{thm:FCT} and
estimate~\eqref{eq:diff-est-r3barterm-v1} in~\eqref{eq:diff-est-r3term-v1}, it
follows that \(R_3 = o({\Delta v})\).
  A similar proof establishes \(R_{4} = o(\Delta v)\).
  Therefore \(\sum_{i=1}^4 R_i = o({\Delta v})\) and Theorem~\ref{thm:gradient} is proved.
\endgroup

\providecommand{\bysame}{\leavevmode\hbox to3em{\hrulefill}\thinspace}
\providecommand{\MR}{\relax\ifhmode\unskip\space\fi MR }
\providecommand{\MRhref}[2]{%
  \href{http://www.ams.org/mathscinet-getitem?mr=#1}{#2}
}
\providecommand{\href}[2]{#2}

\end{document}